\documentclass[reqno]{amsart}
\usepackage{amsmath,amsthm,amscd,amssymb,amsfonts, amsbsy}
\usepackage{latexsym, color, enumerate}
\usepackage{pxfonts}

\theoremstyle{plain}
\newtheorem{theorem}[equation]{Theorem}
\newtheorem{lemma}[equation]{Lemma}
\newtheorem{corollary}[equation]{Corollary}
\newtheorem{proposition}[equation]{Proposition}

\theoremstyle{definition}
\newtheorem{definition}[equation]{Definition}

\theoremstyle{remark}
\newtheorem{remark}[equation]{Remark}

\newcommand{\M}{{\mathcal M}}

\numberwithin{equation}{section}

\newcommand{\bN}{\mathbb{N}}

\newcommand{\bR}{\mathbb{R}}

\providecommand{\norm}[1]{\lVert#1\rVert}

\begin{document}

\title[Singular degenerate elliptic and parabolic equations in non-divergence form]{Weighted mixed-norm $L_p$-estimates for elliptic and parabolic equations in  non-divergence form with singular degenerate coefficients}

\author[H. Dong]{Hongjie Dong}
\address[H. Dong]{Division of Applied Mathematics, Brown University,
182 George Street, Providence, RI 02912, United States of America}
\email{Hongjie\_Dong@brown.edu}

\author[T. Phan]{Tuoc Phan}
\address[T. Phan]{Department of Mathematics, University of Tennessee, 227 Ayres Hall,
1403 Circle Drive, Knoxville, TN 37996-1320 }
\email{phan@math.utk.edu}

\thanks{H. Dong was partially supported by the NSF under agreement  DMS-1600593. T. Phan was partially supported by the Simons Foundation, grant \# 354889.}

\subjclass[2010]{35K65, 35K67, 35J70, 35J75, 35D35, 35D10,  	35B45}
\keywords{Elliptic and parabolic equations in non-divergence forms,  Singular and degenerate coefficients,  Weighted mixed-norm estimates, Calder\'{o}n-Zygmund estimates, Weighted Sobolev spaces}
\begin{abstract}
In this paper, we study both elliptic and parabolic equations in non-divergence form with singular degenerate coefficients. Weighted and mixed-norm $L_p$-estimates and solvability are established under some suitable partially weighted \textup{BMO} regularity conditions on the coefficients.
When the coefficients are constants, the operators are reduced to extensional operators which arise in the study of fractional heat equations and fractional Laplace equations. Our results are new even in this setting and in the unmixed case. For the proof, we establish both interior and boundary Lipschitz estimates for solutions and for higher order derivatives of solutions to homogeneous equations. We then employ the perturbation method by using the Fefferman-Stein sharp function theorem, the Hardy-Littlewood maximum function theorem, as well as a weighted Hardy's inequality.
\end{abstract}

\maketitle

\today
\section{Introduction and main results}
Let $d \in \bN$ and $\bR_+ = (0, \infty)$, $\bR^d_+ = \bR^{d-1} \times \bR_+$ be the upper half space, $T >0$, and $\Omega_T = (-\infty, T] \times \bR^{d}$.
Let $\mathcal{L}$ be the second order partial differential operator in non-divergence form on the upper half space with Gru\v{s}in type singular coefficients
\begin{equation*} 
\mathcal{L}u = a_{ij}(t, x) D_{ij}u  + \frac{\alpha}{x_d} a_{dj}(t, x) D_j u, \quad  (t, x) = (t, x', x_d) \in
\Omega_T, 
\end{equation*}
where $\alpha \in (-1,1)$ is a fixed number, and $(a_{ij})_{d \times d} :\Omega_T \rightarrow \bR^{d\times d}$ is a given bounded measurable matrix-valued function which satisfies the following uniform ellipticity condition with the ellipticity constant $\nu>0$
\begin{equation} \label{ellipticity}
\nu |\xi|^2 \leq a_{ij}(t, x)\xi_i \xi_j, \quad |a_{ij}(t, x)| \leq \nu^{-1},
\end{equation}
for any $\xi = (\xi_1, \xi_2, \ldots, \xi_n) \in \bR^d$ and $(t,x) \in \Omega_T$.
Also let $a_0, c_0: \Omega_T \rightarrow \bR$ be measurable functions satisfying
\begin{equation}  \label{a-b.zero0}
\nu \leq a_0(t,x), \ c_0(t,x) \leq \nu^{-1} \quad \text{for a.e.} \quad (t,x) \in \Omega_T.
\end{equation}

We are interested in the regularity and solvability in weighted and mixed-norm Sobolev space $W^{1,2}_{q,p}(\Omega_T,  \omega x_d^{\alpha}\, dxdt)$ of the following type of parabolic equations with singular coefficients 
\begin{equation} \label{main-eqn}
\left\{
\begin{array}{ccc}
a_0 u_t - \mathcal{L} u +\lambda  c_0 u & = & f(t, x) \quad    \\
\displaystyle{\lim_{x_d \rightarrow 0^+}  x_d^{\alpha} a_{dj}(t, x', x_d) D_j u (t, x', x_d)} & = & 0
\end{array} \quad \text{in} \quad \Omega_T. \right.
\end{equation}
When the coefficients $a_{ij}, c_0$ and the data $f$ are time independent, we also study the corresponding elliptic equations
\begin{equation} \label{elli-main-eqn}
\left\{
\begin{array}{ccc}
- \mathcal{L} u +\lambda c_0 u & = & f( x)   \\
\displaystyle{\lim_{x_d \rightarrow 0^+}  x_d^{\alpha} a_{dj}( x', x_d) D_j u (x', x_d)} & = & 0
\end{array}  \quad \text{in} \quad \bR^d_+. \right.
\end{equation}
Observe that by dividing both sides by $a_0$ or $c_0$ and replacing $\nu$ with $\nu^2$, we may assume that either $a_0\equiv 1$ or $c_0\equiv 1$.

In the special case when $(a_{ij})_{d\times d}$ is the $d\times d$-identity matrix and $a_0\equiv c_0\equiv 1$, the operator $\mathcal{L}$ appears in the study of fractional Laplace equations and fractional heat equations; see \cite{BG17, Caffa-Sil, Caffa-Stinga, Stinga}. In general, the class of equations with singular coefficients as in \eqref{main-eqn} and \eqref{elli-main-eqn} appear in many other problems such as in mathematical finance (see \cite{Pop-1, Pop-2}), in geometric equations (see \cite{Grushin, Gr-Sa}) and in mathematical biology (see \cite{EM}).  Note also that if $\alpha =0$, the considered equations are reduced to the usual second-order elliptic/parabolic equations in non-divergence form with the Neumann or oblique boundary condition. In this case, the existence, uniqueness and regularity estimates in Sobolev spaces 
have been studied extensively in classical literature such as \cite{FHL, Krylov-book, Me}.

The main goals in this paper are to develop regularity estimates in weighted and mixed-norm Sobolev spaces and to study existence and uniqueness of solutions of both parabolic and elliptic equations \eqref{main-eqn} and \eqref{elli-main-eqn}. In addition to this, for homogeneous equations and with nicer coefficient $(a_{ij})$, we also develop local interior and boundary $L_\infty$-estimates for solutions and their higher order derivatives. This paper therefore provide fundamental theory for the study of other problems arising in nonlocal fractional heat equations, fractional Laplace equations, mathematical finance, mathematical biology, and problems in geometric equations.  To the best of our knowledge, this paper is the first of its kind that provides regularity theory in Sobolev spaces for the equations in non-divergence form with measurable and singular coefficients. Even when the coefficients are constants, and in the un-weighted case, the results of the paper are completely new.

To put our study in perspective, let us recall several related results. First of all, due to its interest in calculus of variation,  the study of elliptic and parabolic equations with singular-degenerate coefficients appeared in many classical papers, for example \cite{Fabes-1, Fabes, Str, Stredulinsky, Smith-Stredulinsky, Chia-Sera, Chiarenza-1, Chiarenza-Serapioni}.  Recently, due to various interests, singular-degenerate equations are also studied in \cite{Caffa-Sil, Caffa-Stinga, Pop-1, Pop-2, Grushin, Gr-Sa, EM, Stinga}. In those mentioned paper,  Harnack inequalities and H\"{o}lder regularity theory are established.  On the other hand, regularity theory in Sobolev spaces for elliptic and parabolic equations in divergence form with singular-degenerate coefficients was only studied very recently; see \cite{CMP, Dong-Phan, MP}.
However, regularity theory in Sobolev spaces for equations in non-divergence form with singular-degenerate coefficients is still not part of the literature. Also, even for homogeneous equations with constant coefficients, Lipschitz regularity theory estimates for solutions of \eqref{main-eqn}-\eqref{elli-main-eqn} and for their higher order derivatives are not yet established. The goals of this paper are to establish those fundamental estimates, and then use them to obtain regularity theory in weighted and mixed-norm Sobolev spaces for solutions of \eqref{main-eqn} as well as of \eqref{elli-main-eqn}.

To state the main results of the paper, we introduce some definitions and notation used in the paper.  For each $x= (x', x_{d}) \in \bR^{d}_+$ and $\rho >0$, we write
\[
D_\rho(x) = B_\rho'(x') \times (x_{d} -\rho, x_{d} + \rho), \quad D_\rho^+(x) =  B_\rho'(x') \times ((x_{d} -\rho)^+, x_{d} + \rho),
\]
where $B_\rho'(x')$ is the ball in $\bR^{d-1}$  of radius $\rho$ and centered at $x'$, and $a^+ = \max\{0, a\}$ for every real number $a$. Moreover, for each $z = (t, x)=(z',x_d) \in \bR \times \bR^n$, the parabolic cylinder and half-parabolic cylinder in $\bR \times \bR^n$ of radius $\rho$ centered at $z$ are denoted by
$$Q_\rho(z) = (t - \rho^2, t] \times D_\rho(x), \quad Q_\rho^+(z) = (t- \rho^2, t] \times D^+_\rho(x).
$$
When $z =(0,0)$, we use the abbreviations
\[
B'_\rho = B'_\rho (0),\quad D_\rho = D_\rho(0),\quad Q_\rho = Q_\rho(0), \quad \text{etc}. 
\]
For every $z_0 = (t_0, x_0' , x_{d0}) \in \overline{\Omega}_T$, $\rho >0$, 
and a measurable function $f$ defined in $Q_{\rho}^+(z_0)$, we denote
\[
\fint_{Q_{\rho}^+(z_0)} f(z) \ \mu(dz)=\frac{1}{\mu(Q_{\rho}^+(z_0))} \int_{Q_{\rho}^+(z_0)}  f(z) \ \mu(dz)
\]
with
\[
\mu(Q_{\rho}^+(z_0)) = \int_{Q_{\rho}^+(z_0)} \ \mu(dz),\]
and we also denote
\begin{equation} \label{average-a}
 [f]_{\rho, z_0}(x_d)=\fint_{Q'_\rho(z_0')} f(t, x', x_d)\, dx'dt ,
\end{equation}
where $z_0' = (t_0, x_0')$ and $Q'_\rho(z_0') = (t_0 -\rho^2, t_0] \times B_\rho'(x_0')$, the parabolic cylinder in $\bR \times \bR^{d-1}$. Moreover, the partial weighted mean oscillations of the matrix $a= (a_{ij})_{i,j=1}^d$, the coefficients $a_0$, and $c_0$ is denoted by
\begin{align} \label{a-sharp}
a_{\rho}^\#(z_0) &= \max_{ \ i, j = 1,2,\ldots, d}\fint_{Q_\rho^+(z_0)}\Big|a_{ij}(z)-[a_{ij}]_{\rho, z_0}(x_d)\Big|\,\mu(dz)\nonumber\\
&\quad +\fint_{Q_\rho^+(z_0)}\Big(\big|a_0(z)-[a_0]_{\rho, z_0}(x_d)\big|+\big|c_0(z)-[c_0]_{\rho, z_0}(x_d)\big|\Big)\,\mu(dz)
\end{align}
for $z_0\in \overline{\Omega_T}$,
where $\mu(y) = |y|^\alpha$ for $y \in \bR \setminus\{0\}$ and $\mu(dz) = \mu(x_d) \ dx dt$.

For $p,q\in (1,\infty)$ and for given weights $\omega_0=\omega_0(t)$ and $\omega_1=\omega_1(x)$, we denote the weighted mixed-norm Lebesgue space on $\Omega_T$ by $L_{q,p}(\Omega_T,\omega\,d\mu)$ equipped with the norm
\begin{equation*}
\|f\|_{L_{q,p}(\Omega_T, \omega\, d\mu)}=\left(\int_0^T\big(\int_{\bR^d_+} |f|^p \omega_1(x)\,\mu(dx)\big)^{q/p}\omega_0(t)\,dt\right)^{1/q},
\end{equation*}
where $\mu(dx) = \mu(x_d) \ dx$ and $\omega(t,x)=\omega_0(t)\omega_1(x)$.
Similarly,  $W^{1,2}_{q,p}(\Omega_T, \omega\, d\mu)$ denotes the weighted mixed-norm Sobolev space equipped with the norm
\[
\norm{u}_{W^{1,2}_{q, p} (\Omega_T, \omega\, d\mu)} = \norm{u}_{L_{q, p}(\Omega_T, \omega\,d\mu)} + \norm{u_t}_{L_{q, p} (\Omega_T, \omega\, d\mu)} + \norm{Du}_{L_{q, p} (\Omega_T, \omega\, d\mu)} + \norm{D^2u}_{L_{q, p}(\Omega_T, \omega\, d\mu)}.
\]
When $p =q$ and $\omega \equiv 1$, we write $L_p(\Omega_T, d\mu) = L_{p, p}(\Omega_T, d\mu)$  and $W^{1,2}_p(\Omega_T, d\mu) = W^{1,2}_{p,p}(\Omega_T, d\mu)$.

In addition to the ellipticity condition \eqref{ellipticity}, we assume that the coefficient matrix $a= (a_{ij})_{i,j=1}^d : \Omega_T  \rightarrow \bR^{d\times d}$ satisfies the structural condition
\begin{equation} \label{extension-type-matrix}
a_{dj}(t,x) =0, \quad j = 1,2,\ldots, d-1.
\end{equation}
We would like to point out that this structural condition \eqref{extension-type-matrix} holds for a large class of equations arising in other problems such as \cite{Caffa-Sil, Pop-1, EM, Grushin, Gr-Sa}.

We now refer readers to Section \ref{Prelim}, where we recall and discuss the Muckenhoupt $A_p$-class of weights,  and the $M_{p}$-class of weights in which a certain type of weighted Hardy's inequality holds. 
Our first main result is the following theorem on the existence, uniqueness, and regularity estimates for solutions of \eqref{main-eqn}.
\begin{theorem}
\label{para-main.theorem}  Let $\nu \in (0,1)$, $T\in (-\infty,\infty]$, $p, q, K \in (1, \infty)$, $\alpha \in (-1,1)$, and $\rho_0 >0$. Then
there exist $\delta = \delta (d, \nu, p, q, \alpha, K) >0$ sufficiently small, and $\lambda_0(\nu,d,p,q,\alpha, K, \rho_0)>0$ sufficiently large such that the following statement holds. Suppose that $\omega_0\in A_q(\bR)$, $\omega_1 \in A_p(\bR^{d-1})$, $\omega_2 \in A_p(\bR_+, \mu)) \cap M_{p}(\mu)$ with
$$
[\omega_0]_{A_q(\bR)}, \quad [\omega_1]_{A_p(\bR^{d-1})} ,  \quad [\omega_2]_{A_p(\bR_+, \mu)}, \quad [\omega_2]_{M_{p}(\mu)} \le K.$$
Suppose also that \eqref{ellipticity}, \eqref{a-b.zero0}, \eqref{extension-type-matrix} hold and
\begin{equation} \label{para-VMO}
a^{\#}_{\rho}(z_0) \leq \delta, \quad \forall \ \rho \in (0, \rho_0), \quad \forall \ z_0 \in  \overline{\Omega}_T.
\end{equation}
Then for any $f \in L_{q,p}(\Omega_T, \omega\, d\mu)$ and $\lambda \ge \lambda_0$, there exists a unique solution $u \in W^{1,2}_{q,p}(\Omega_T,  {\omega}\ d\mu)$ of \eqref{main-eqn}. Moreover, it holds that
\begin{equation} \label{main-para.est}
\begin{split}
& \norm{u_t}_{L_{q,p}(\Omega_T, \omega\, d\mu)} + \norm{D^2u}_{L_{q,p}(\Omega_T, \omega\, d\mu)} + \sqrt{\lambda} \norm{Du}_{L_{q,p}(\Omega_T, \omega\, d\mu)} + \lambda \norm{u}_{L_{q, p}(\Omega_T, \omega\, d\mu)}\\
& \leq  C(\nu, d, p, q, \alpha, K) \norm{f}_{L_{q,p}(\Omega_T, \omega\, d\mu)},
\end{split}
\end{equation}
where $\omega(t,x) = \omega_0(t) \omega_1(x') \omega_2(x_d)$ for all $(t,x) = (t, x' , x_d) \in \Omega_T$, and $d\mu = x_d^\alpha \ dx dt$.
\end{theorem}

In the case when the coefficients and data are time-independent,  for each $x_0 \in \bR^d_+$, we define $[a_{ij}]_{\rho, x_0}$, $[c_0]_{\rho, x_0}$, and $a_{\rho}^{\#}(x_0)$ as in \eqref{average-a} and \eqref{a-sharp} by replacing the parabolic cylinders $Q_{\rho}^+(z_0)$ and {$Q'_\rho(z_0')$} by balls $D_{\rho}^+(x_0)$ and $B'_\rho(x_0')$, respectively. As a consequence of Theorem \ref{para-main.theorem}, we have the following result for the elliptic equation \eqref{elli-main-eqn} with singular coefficients.

\begin{theorem} \label{elli-main.theorem} Let $\nu \in (0,1), p,K \in (1, \infty)$,  $\alpha \in (-1,1)$, and $\rho_0>0$. Then, there exist $\delta = \delta (d, \nu, p, \alpha, K) >0$ sufficiently small, and $\lambda_0(\nu,d,p,\alpha, K,\rho_0)$ sufficiently large such that the following statement holds.  Suppose that  $\omega_1 \in A_p(\bR^{d-1}), \omega_2 \in A_p(\bR_+, \mu) \cap M_{p}(\mu)$ with
$$
[\omega_1]_{A_p(\bR^{d-1})} ,  \quad [\omega_2]_{A_p(\bR_+, \mu)}, \quad [\omega_2]_{M_{p}(\mu)} \le K.
$$
Suppose also that  \eqref{ellipticity}, \eqref{a-b.zero0}, \eqref{extension-type-matrix} hold on $\bR_+^d$ and
\begin{equation} \label{elli-VMO}
a^{\#}_{\rho}(x_0) \leq \delta, \quad \forall \ \rho \in (0, \rho_0), \quad \forall \ x_0 \in  \overline{\bR_+^d}.
\end{equation}
Then for any $f \in L_{p}(\bR^d_+, \omega\, d\mu)$ and for $\lambda \geq \lambda_0$, there exists a unique solution $u \in W^{2}_{p}(\bR^d_+, \omega\, d\mu)$ of \eqref{elli-main-eqn}. Moreover, it holds that
\begin{equation} \label{main-ell.est}
\begin{split}
&  \norm{D^2u}_{L_{p}(\bR_+^d, \omega\, d\mu)} + \sqrt{\lambda} \norm{Du}_{L_{p}(\bR_+^d, \omega\,d\mu)} + \lambda \norm{u}_{L_{p}(\bR_+^d,  \omega\, d\mu)} \\&  \leq  C(\nu, d, p, \alpha, K) \norm{f}_{L_{p}(\bR_+^d, \omega\, d\mu)},
 \end{split}
\end{equation}
where $\omega(x) = \omega_1(x') \omega_2(x_d)$ for all $x = (x', x_d) \in \bR^{d-1} \times \bR_+$, and $d\mu = x_d^\alpha \ dx$.
\end{theorem}

\begin{remark} The following points regarding Theorems \ref{para-main.theorem} and \ref{elli-main.theorem} are worth mentioning
\begin{itemize}
\item[\textup{(i)}] If $a_{ij}$ depends only on the $x_d$-variable,  then the conditions \eqref{para-VMO} and \eqref{elli-VMO} always hold. See Theorem \ref{W-2-p.constant.eqn} below. 
\item[\textup{(ii)}] Theorems \ref{para-main.theorem} and \ref{elli-main.theorem} hold when $\omega_2(y) = y^\beta$ for $y \in (0, \infty)$ and with $\beta \in (-\alpha -1, (\alpha +1) (p-1))$. When $\alpha=0$, $p=q$, $\omega \equiv 1$, and $a_{ij}$ are VMO in $x$ and merely measurable in $t$, a similar result was proved recently in \cite{KDZ16}.
\item[\textup{(iv)}]By using a change of variables $y_j=x_j-\lambda_{j} x_d,j=1,2,\ldots,d-1$ and $y_d=x_d$, we can relax the condition \eqref{extension-type-matrix} and replace it with the condition that $a_{dj}=\lambda_{j}a_{dd}$ for $j=1,2,\ldots,d-1$ with constants $\lambda_{j}$.
\item[\textup{(v)}] By localization using the cut-off function $\phi(x) = \phi_1(x') \phi_2(x_d)$ with some standard cut-off functions $\phi_1, \phi_2$ and $x =(x', x_d) \in \bR^{d-1}\times \bR$, local weighted mixed-norm boundary estimates can also be obtained from Theorems \ref{para-main.theorem} and \ref{elli-main.theorem}.
\end{itemize}
\end{remark}

Theorems \ref{para-main.theorem} and \ref{elli-main.theorem} appear for the first time for equations with singular-degenerate coefficients in non-divergence form. We also emphasize again that even in the un-weighted case and when the coefficients $a_{ij}$ depend only on the $x_d$-variable, these theorems are new. In this case, our results give $W^{2,p}$-regularity estimates for the extension problem of the fractional Laplace and fractional heat equations. The results therefore provide the Sobolev counterpart of the $C^\alpha$ and $C^{1,\alpha}$ regularity theory established in \cite{Caffa-Sil, Caffa-Stinga} and in many other papers.  We also note that for equations in divergence form,  similar $W^{1,p}$-regularity estimates are established in \cite{Dong-Phan} for parabolic equations and in \cite{MP} for elliptic equations. When $\alpha =0$, the equations \eqref{main-eqn} and \eqref{elli-main-eqn} are reduced to uniformly elliptic and parabolic equations. These equations are studied extensively in many papers (see \cite{Dong-Kim-1, Dong-Kim, Krylov-book, Kim-Krylov, Kim-Krylov-1, Krylov-07}) and similar results are established. The results of this paper can be considered as a further intensive development in this line of research.

Our proof is based the perturbation method using the Fefferman-Stein sharp functions that was introduced in \cite{Krylov-book} and developed in \cite{Dong-Kim-18, Dong-Kim-1, Dong-Kim, Kim-Krylov, Kim-Krylov-1, Krylov-07}.  See also \cite{Dong-Phan} for a different approach which uses level set estimates for equations in divergence form.
To implement the perturbation technique, we first establish both local interior and boundary $L_\infty$-estimates for solutions and for higher order derivatives of solutions to homogeneous equations.  These results are also topics of independent interests and they can be useful for other purposes.  A weighted Hardy type inequality also plays an important role in our proof.

We conclude this paper by giving an outline of the paper. In the next section, Section \ref{Prelim}, we recall several definitions on weights and weighted inequalities. As an intermediate step, in  Section \ref{simple-coeffients} we study equations with coefficients depending only on the $x_d$-variable. In this section,  local interior and boundary $L_\infty$-estimates for solutions and for higher order derivatives of solutions are established.  A result on unmixed weighted Sobolev estimates for non-homogeneous equations,  which is a special case of Theorem \ref{para-main.theorem},  is also proved in this section. Equations with singular-degenerate partially weighted \textup{BMO} coefficients are studied in Section \ref{VMO-section}, which is concluded with the proofs of Theorems \ref{para-main.theorem} and \ref{elli-main.theorem}.

\section{Preliminaries on analysis of weights and weighted inequalities} \label{Prelim}

In this section, we recall several weighted estimates needed for the paper.
\subsection{Weighted Hardy's inequality}
Let us begin with the following definition.
\begin{definition} Let $\alpha \in (-1,1)$ and $p \in (1,\infty)$, we say that the weight $\omega: \bR_+ \rightarrow \bR_+$ is in $M_{p}(\mu)$ if
\[
[\omega]_{M_{p}(\mu)} : = \sup_{r>0} \left[\int_r^\infty y^{- p(\alpha +1)} \omega(y) \ \mu(dy) \right]^{\frac 1 p} \left[ \int_0^r \omega(y)^{-\frac{1}{p-1}} \ \mu(dy) \right]^{1-\frac{1}{p}} < \infty,
\]
where $\mu(y) = y^\alpha$ for $y \in \bR_+$.
\end{definition}

See Remark \ref{A-B-weight} below for examples of nontrivial $M_p(\mu)$ weights.  The following lemma is a consequent of \cite[Theorem 1]{Mucken-72}.
\begin{lemma}[Weighted Hardy's inequality] \label{wei-hardy-lemma} Let $\alpha \in (-1,1)$, $p \in (1,\infty)$, $\mu(y) = y^\alpha$ for $y \in \bR_+$,  and let $\omega :\bR_+ \rightarrow \bR_+$ be measurable. Then, there is $C  >0$ such that
\begin{equation} \label{Hardy-ineq}
\int_0^\infty \left( \frac{1}{y^{\alpha+1}} \int_0^y  f(s) \ \mu(ds) \right)^p  \omega(y) \ \mu(dy) \leq C \int_0^\infty f(y)^p  \omega(y) \ \mu(dy),
\end{equation}
for any non-negative $f \in L_p(\bR_+, \omega d\mu)$ if and only if $\omega \in M_{p}(\mu)$. Moreover, the constant $C$ depends only on $p$ and $[\omega]_{M_{p}(\mu)}$.
\end{lemma}
\begin{proof} Let $F(y) = y^\alpha f(y)$ for $y \in \bR_+$. Then, \eqref{Hardy-ineq} is equivalent to
\[
\int_0^\infty \left( y^{\alpha(\frac{1}{p} -1) -1} \omega^\frac{1}{p}(y) \int_0^y F(s) \ ds \right)^p dy \leq C \int_0^\infty \Big[y^{\alpha(\frac{1}{p} -1)} \omega^{\frac{1}{p}} F(y) \Big]^p dy.
\]
The assertion of the lemma then follows directly from \cite[Theorem 1]{Mucken-72}.
\end{proof}
We now introduce the following simple but important lemma in the paper.
\begin{lemma} \label{D-dd-control} Let $\nu \in (0,1), \alpha \in (-1,1)$ and $p ,q \in (1, \infty)$.  Let $\omega (t,x) = \omega_0(t) \omega_1(x') \omega_2(x_d)$ where $\omega_0, \omega_1$ are any weights, and $\omega_{2} \in M_{p}(\mu)$. Suppose that \eqref{ellipticity} and \eqref{extension-type-matrix} hold.  Then for 
any $R \in (0, \infty]$, if $u \in W^{1,2}_{ q,p}(Q_{R}^+,\omega\,d\mu)$ is a solution of
\[
\left\{
\begin{array}{ccc}
u_t -\mathcal{L} u + \lambda u & = & f  \\
\displaystyle{\lim_{x_d\rightarrow 0^+}} x_d^\alpha D_du(t,x', x_d) & =& 0
\end{array} \quad \text{in} \quad Q_{R}^+
 \right.
 \]
with some $\lambda \ge 0$ and $f \in L_{q, p}(Q_R^+(\hat{z}),  \omega\, d\mu)$,  it holds that
\begin{equation*} 
\begin{split}
&\|D_{d}u/x_d\|_{L_{q,p}(Q_R^+,  \omega\, d\mu)}+ \|D_{d}^2u\|_{L_{q,p}(Q_R^+,  \omega\, d\mu)} \leq  C(d,\nu,p, [\omega_2]_{M_{p}(\mu)}) \left[\|u_t\|_{L_{q,p}(Q_R^+,  \omega\, d\mu)} \right. \\
 & \quad \quad  \left. + \|DD_{x'}u\|_{L_{q,p}(Q_R^+,  \omega\, d\mu)} +  \lambda \|u\|_{L_{q,p}(Q_R^+,  \omega\, d\mu)} + \|f\|_{L_{q,p}(Q_R^+,  \omega\, d\mu)} \right],
\end{split}
\end{equation*}
where $\mu(dx) = x_d^\alpha \ dx$.
\end{lemma}
\begin{proof} Note that
\[
D_d(x^\alpha_d D_d u) =x^\alpha_d \Big[ \frac{\alpha}{x_d} D_{d}u + D_{d}^2 u \Big].
\]
From this, the conditions \eqref{ellipticity} and \eqref{extension-type-matrix}, and the equation of $u$, we obtain
\[
|D_{d}(x^\alpha_d D_du)| \leq C(d, \nu) x_d^\alpha F, \quad \text{where} \quad F= |f| + \lambda |u| + |u_t| + |DD_{x'}u| .
\]
This estimate together with the boundary condition of $x_d^\alpha D_du$ yields
\[
\begin{split}
& \int_{D_R^+} \Big| \frac{ D_d u}{x_d}\Big |^p x_d^{\alpha} \omega_1(x')\omega_2(x_d) \ dx \\
& \le C(d,\nu)\int_{D_R^+} \left(\frac  1 {x_d^{\alpha+1}}
\int_{0}^{x_d}s^\alpha |F(z', s)| \ ds\right)^p x_d^{\alpha} \omega_1(x')  \omega_2(x_d)\ dx.
\end{split}
\]
Note that from the weighted Hardy's inequality in Lemma \ref{wei-hardy-lemma}, we have
\[
\begin{split}
& \int_0^{R} \left(\frac  1 {x_d^{\alpha+1}}
\int_{0}^{x_d}s^\alpha |F(z', s)| \ ds\right)^p x_d^{\alpha}  \omega_2(x_d) \ dx_d \\
& \le C(p, [\omega_2]_{M_{p}(\mu)}) \int_{0}^{R}  |F(z', x_d)|^p x_d^\alpha \omega_2(x_d) \ dx_d.
\end{split}
\]
From the above two inequalities, it follows that
\[
\begin{split}
\int_{D_R^+} \Big| \frac{D_d u}{x_d}\Big |^p x_d^{\alpha} \omega_1(x') \omega_2(x_d) \ dx  \leq C(d,\nu,p,[\omega_2]_{M_{p}(\mu)}) \int_{D_R^+} |F(z', x_d)|^p x_d^{\alpha} \omega_1(x')\omega_2(x_d) \ dx.
\end{split}
\]
This and an integration in the time variable imply the estimate of $D_du/x_d$.
To estimate $D_{d}^2 u$, we use the equation again to derive
\[
|D_{d}^2u|  \le C(d,\nu) \Big[ F + |D_du/x_d|\Big].
\]
From this, the estimate of $D_{d}^2 u$  follows and the proof is completed.
\end{proof}
\subsection{Muckenhoupt weights and weighted Fefferman-Stein theorem}

We start with recalling the definition of classes of Muckenhoupt weights which was introduced in \cite{Muck}.
\begin{definition} Let $\alpha \in (-1,1)$ and $\mu(y) = |y|^\alpha$ for $y \in \bR\setminus \{0\}$. For each $p \in [1, \infty)$ and each $d \in \bN$, a locally integrable function $\omega : \bR^d \rightarrow \bR_+$ is said to be in $A_p(\bR^d, \mu)$ Muckenhoupt class of weights if and only if $[\omega]_{A_p(\bR^d, \mu)} < \infty$, where
\begin{equation*}
[\omega]_{A_p(\bR^d, \mu)} = \left\{
\begin{array}{lll}
& \displaystyle{\sup_{\rho >0,x \in \bR^d} \left[\fint_{B_\rho(x)} \omega(y)\, \mu(dy) \right]\left[\fint_{B_\rho(x)} \omega(y)^{\frac{1}{1-p}}\, \mu(dy) \right]^{p-1}} & \,\, \text{if}\ p \in (1, \infty), \\
& \displaystyle{\sup_{\rho >0,x \in \bR^d} \left[\fint_{B_\rho(x)} \omega(y)\, \mu(dy) \right] \norm{\frac{1}{\omega}}_{L^\infty(B_\rho(x))}}, & \,\, \text{if} \ p =1.
\end{array} \right.
\end{equation*}
If $\mu$ is a Lebesgue measure (i.e. $\alpha =0$), we just simply write $A_p(\bR^d) = A_p(\bR^d, \mu)$. Observe that if $\omega \in A_p(\bR)$, then $\tilde{\omega} \in A_p(\bR^d)$ with $[\omega]_{A_p(\bR)} = [\tilde{\omega}]_{A_p(\bR^d)}$, where $\tilde{\omega}(x) = \omega(x_n)$ for $x = (x', x_n) \in \bR^d$. Sometimes, if the context is clear, we neglect the spacial domain and only write $\omega \in A_p$.
\end{definition}
\noindent
\begin{remark} \label{A-B-weight} Let $\alpha \in (-1,1)$, $p \in (1,\infty)$, and $\beta\in (-\alpha-1, (\alpha+1)(p-1))$. Let $\omega(y) = y^\beta$ for $y \in (0,\infty)$. Then, it is easily seen that $\omega \in M_{p}(\mu) \cap A_p(\bR_+, \mu)$.
\end{remark}

Let us denote the collection of parabolic cylinders in $\Omega_T = (-\infty, T] \times \bR^d \times \bR_+$ by
\[
\mathcal{Q} = \{ Q_\rho^+(z): \rho >0, z \in \Omega_T \}.
\]
For any locally integrable function $f$ defined in $\Omega_T$, the Hardy-Littlewood maximal function of  $f$ is defined by
\[
\mathcal{M}(f)(z) = \sup_{Q \in \mathcal{Q}, z \in Q}\fint_{Q} |f(\xi)| \ \mu(d\xi).
\]
Moreover, the Fefferman-Stein sharp function of $f$ is defined by
\begin{equation} \label{sharp-def}
f^{\#}(z) =  \sup_{Q \in \mathcal{Q},z  \in Q} \fint_{Q}|f(\xi) - (f)_Q| \ \mu(d\xi),
\quad \text{where}\,\,(f)_Q=\fint_{Q} |f(\xi)| \ \mu(d\xi).
\end{equation}
 The following version of weighted and mixed-norm Fefferman-Stein theorem and Hardy-Littlewood maximal function theorem are established in \cite{Dong-Kim-18} and they are needed in the paper.
\begin{theorem}  \label{FS-thm} Let $p, q \in (1,\infty)$. Suppose that $\omega_0\in A_q(\bR)$, $\omega_1 \in A_p(\bR^{d-1})$, $\omega_2 \in  A_p(\bR_+, \mu)$ with
$$
[\omega_0]_{A_q}, \quad [\omega_1]_{A_p} ,  \quad [\omega_2]_{A_p(\bR_+, \mu)}\le K.$$
Then, for every $f \in L_{q, p}(\Omega_T, \omega\, d\mu)$,
\begin{equation} \label{Maximal-L-p}
\begin{split}
& \|f\|_{L_{q, p}(\Omega_T,  \omega\, d\mu)} \leq C(d, q, p, K) \| f^{\#}\|_{L_{q,p}(\Omega_T,  \omega\, d \mu)} \quad \text{and} \quad \\
& \|\mathcal{M}(f)\|_{L_{q,p}(\Omega_T, \omega\, d\mu)} \leq C(d, q, p, K) \|f\|_{L_{q, p}(\Omega_T,  \omega \ d \mu)}.
\end{split}
\end{equation}
\end{theorem}
\noindent

\section{Equations with coefficients depending only on $x_d$-variable} \label{simple-coeffients}

This section studies equations with singular coefficients that depend only on the $x_d$-variable.  
We establish Lipschitz estimates for solutions to  homogeneous equations and for their higher order derivatives. From these Lipschitz estimates, we derive mean oscillation estimates for certain derivatives of the solutions.  We then state and prove $W^{1,2}_p$-regularity for solutions to non-homogeneous equations. These results
will be used in the next section when we study equations with singular partially weighted VMO coefficients.

Let $T \in (-\infty, \infty]$ and let $(a_{ij})_{d \times d} : \bR_+ \rightarrow \bR^{d\times d}$ be a bounded matrix-valued function which is uniformly elliptic, and let $\alpha \in (-1,1)$. Also, let $a_0, c_0: \bR_+ \rightarrow \bR$ be measurable functions satisfying
\begin{equation}  \label{a-b.zero}
\nu \leq a_0(x_d), \ c_0(x_d) \leq \nu^{-1} \quad \text{for a.e.} \quad x_d \in \bR_+.
\end{equation}
We denote 
\[
\mathcal{L}_0 u(t, x) =  a_{ij}(x_d) D_{ij} u(t, x', x_d) + \frac{\alpha}{x_d} a_{dj} D_j u(t, x', x_d), \quad (t, x)  = (t, x', x_d) \in \Omega_T,
\]
and consider the equation
\begin{equation} \label{constant-L-0}
\left\{
\begin{array}{ccc}
a_0(x_d) u_t - \mathcal{L}_0 u(t, x)  + \lambda  c_0(x_d) u& = & f(t,x) \\
\displaystyle{\lim_{x_d \rightarrow 0^+} x_d^\alpha a_{dj}(x_d) D_{j} u(t, x', x_d)} & = & 0
\end{array} \quad \text{in} \quad \Omega_T \right.
\end{equation}
where $\lambda > 0$ is fixed. In addition to the assumption that the matrix $(a_{ij})_{i, j =1}^d$ is uniformly elliptic and bounded, we assume that $a_{dj}/a_{dd}, j = 1, 2,\ldots, d-1$ are constant functions and therefore they are independent on $x_d$-variable. 
Dividing both sides of the equation by $a_{dd}$, we may assume that
\begin{equation*} 
a_{dj}(x_d) = a_{dj} \quad \text{and} \quad a_{dd}(x_d) =1, \quad \forall \ x_d \in \bR_+, \quad j = 1, 2,\ldots, d-1.
\end{equation*}
Note also that by the change of variables, $y_j=x_j-a_{dj}x_d,j=1,2,\ldots,d-1$ and $y_d=x_d$, without loss of generality, we may assume that $a_{dj} =0$ for $j = 1,2,\ldots, d-1$ as in \eqref{extension-type-matrix}.  In the remaining part of this section, we assume that
\begin{equation} \label{a-dd.cond}
a_{dj}(x_d) = 0 \quad \text{and} \quad a_{dd}(x_d) =1, \quad \forall \ x_d \in \bR_+, \quad j = 1, 2,\ldots, d-1.
\end{equation}
Observe that under the condition \eqref{a-dd.cond}, there is a hidden divergence structure for the operator $\mathcal{L}_0$. Namely,
\[
x_d^\alpha \mathcal{L}_0 u (t, x) = D_i[x_d^\alpha a_{ij}(x_d) D_j u(t, x) ].
\]
Consequentially,  the equation \eqref{constant-L-0} can be rewritten 
in divergence form as
\begin{equation} \label{div-L-0}
\left\{
\begin{array}{cccl}
x_d^{\alpha} a_0(x_d)u_t - D_i[x_d^\alpha a_{ij}(x_d) D_j u(x',x_d)]  + \lambda x_d^\alpha c_0(x_d) u& = & x_d^\alpha f(t,x),  \\
\displaystyle{\lim_{x_d \rightarrow 0^+} x_d^\alpha D_d u (t,x', x_d)}& = & 0 
\end{array} \quad \text{in} \quad \Omega_T. \right.
\end{equation}
\subsection{Lipschitz estimates for homogeneous equations}  We derive several Lipschitz estimates and oscillation estimates for solutions and their certain derivatives. We start with the following local interior and boundary $L_2$ estimates for solutions. 


\begin{lemma}[Local $L_2$-estimates] \label{loc-L2.est-lemma} Let $\nu \in (0, 1]$, $\lambda \geq 0$, and $\alpha \in (-1,1)$ be fixed, and let $\mu(s) = |s|^{\alpha}$ for  $s \in \bR \setminus\{0\}$. Also, let $\hat{z} = (\hat{t}, \hat{x}', \hat{x}_d) \in \bR \times \bR^{d-1} \times \overline{\bR_+}$, and 
assume that  \eqref{ellipticity}, \eqref{a-b.zero}, and \eqref{a-dd.cond} hold on $((\hat{x}_d -2)^+, \hat{x}_d + 2)$. Then, the following statements hold.
\begin{itemize}
\item[\textup{(i)}] If $\hat{x}_d \geq 2$ and $u \in W^{1,2}_2(Q_{2}(\hat{z}))$ is a solution of
\begin{equation} \label{int-Q4-constant-L-0}
 a_0(x_d)u_t - \mathcal{L}_0 u(t, x) +\lambda c_0(x_d) u   =  0\quad  \text{in}\quad Q_{2}(\hat{z}).
\end{equation}
then \begin{equation*}  
\sup_{\tau \in (\hat{t}-1, \hat{t}] }\int_{D_1(\hat{x}) } u^2(\tau,x) \,\mu(dx)
 + \int_{Q_1(\hat{z})}  (|Du|^2 +\lambda u^2) \,{\mu}(dz) \leq C(d, \nu) \int_{Q_2(\hat{z})} u^2 \,{\mu}(dz),
\end{equation*}
and for any $k, j  \in \mathbb{N} \cup \{0\}$,
\begin{equation*} 
 \int_{Q_1(\hat{z}) }|\partial_t^{j+1}u |^2 \,{\mu}(dz) + \int_{Q_1(\hat{z})} |DD_{x'}^k \partial_t^ju |^2 \,{\mu}(dz) \le C(d,k, j, \nu)\int_{Q_2 (\hat{z})}  (|Du|^2  +\lambda u^2) \,{\mu}(dz).
\end{equation*}
\item[\textup{(ii)}]If $\hat{x}_d =0$ and $u \in W^{1,2}_2(Q_{2}^+(\hat{z}))$ is a solution of
\begin{equation} \label{bdr-Q4-constant-L-0}
\left\{
\begin{array}{ccc}
 a_0(x_d) u_t - \mathcal{L}_0 u(t, x)  + \lambda c_0(x_d) u & = & 0 \\
\displaystyle{\lim_{x_d \rightarrow 0^+} x_d^\alpha  D_d u(t, x', x_{d})} & = & 0
\end{array} \quad \text{in} \quad Q_{2}^+(\hat{z}). \right.
\end{equation}
then \begin{equation*}  
\sup_{\tau \in (\hat{t}-1, \hat{t}] }\int_{D_1^+(\hat{x}) } u^2(\tau,x) \,\mu(dx)
 + \int_{Q_1^+(\hat{z})}  (|Du|^2 +\lambda u^2) \,{\mu}(dz)  \leq C(d, \nu) \int_{Q_2^+(\hat{z})} u^2 \,{\mu}(dz),
\end{equation*}
and for any $k, j  \in \mathbb{N} \cup \{0\}$,
\begin{equation*} 
 \int_{Q_1^+(\hat{z}) }|\partial_t^{j+1}u |^2 \,{\mu}(dz) + \int_{Q_1^+(\hat{z})} |DD_{x'}^k \partial_t^ju |^2 \,{\mu}(dz) \le C(d,k, j, \nu)\int_{Q_2^+(\hat{z})}  (|Du|^2 +\lambda u^2) \,{\mu}(dz).
\end{equation*}
\end{itemize}
\end{lemma}
\begin{proof}
Under the assumption \eqref{a-dd.cond}, both equations in \eqref{int-Q4-constant-L-0} and \eqref{bdr-Q4-constant-L-0} can be written in divergence form as in \eqref{div-L-0}. Therefore, the lemma is a special case of \cite[Lemmas 4.5 and 4.16]{Dong-Phan}.
\end{proof}
As a consequence of Lemma \ref{loc-L2.est-lemma}, we can derive local interior and boundary $L_\infty$-regularity estimates for solutions and their first order derivatives.
\begin{lemma} \label{Lip-est-lemma} \label{Lip-D1.lemma} Under the assumptions of Lemma \ref{loc-L2.est-lemma}, the following statements hold true.
\begin{itemize}
\item[\textup{(i)}] If $\hat{x}_d \geq 2$ and $u \in W^{1,2}_2(Q_{2}(\hat{z}))$ is a solution of \eqref{int-Q4-constant-L-0}, then
\begin{equation*} 
                                    \begin{split}
& \|u\|_{L_\infty(Q_1(\hat{z}))} \leq C(d,\nu, \alpha)  \left( \fint_{Q_{2}(\hat{z})} u^2 \,{\mu} (dz) \right)^{1/2}, \quad \text{and} \\
& \| u_t \|_{L_\infty(Q_1(\hat{z}))}  + \| D u\|_{L_\infty(Q_1(\hat{z}))} \leq C(d, \nu,  \alpha) \left( \fint_{Q_{2}(\hat{z})} (|Du|^2 +\lambda u^2) \,{\mu} (dz) \right)^{1/2}.
\end{split}
\end{equation*}
\item[\textup{(ii)}] If $\hat{x}_d =0$ and $u \in W^{1,2}_2(Q_{2}^+(\hat{z}))$ is a solution of \eqref{bdr-Q4-constant-L-0}, then
\begin{equation*}  
\begin{split}
& \|u\|_{L_\infty(Q_1^+(\hat{z}))} \leq C(d,\nu, \alpha)  \left( \fint_{Q_2^+(\hat{z})} u^2 \,{\mu} (dz) \right)^{1/2}, \quad \text{and} \\
& \| u_t \|_{L_\infty(Q_1^+(\hat{z}))}  + \| D u\|_{L_\infty(Q_1^+(\hat{z}))} \leq C(d, \nu,  \alpha) \left( \fint_{Q_2^+(\hat{z})} (|Du|^2 +\lambda u^2) \,{\mu} (dz) \right)^{1/2}.
\end{split}
\end{equation*}
\end{itemize}
\end{lemma}
\begin{proof}
Again, under the assumption \eqref{a-dd.cond}, both equations in \eqref{int-Q4-constant-L-0} and \eqref{bdr-Q4-constant-L-0} can be written in divergence form as in \eqref{div-L-0}. Hence, the lemma follows from \cite[Propositions 4.7 and 4.23]{Dong-Phan}.
\end{proof}
\begin{remark} \label{remark-1} Under the assumptions in Lemma \ref{Lip-est-lemma}, we can use Lemma \ref{loc-L2.est-lemma} with some slight modification to improve the Lipschitz estimates of the solution $u$ as
\[
\| u \|_{L_\infty(Q_1^+(\hat{z}))}+\| u_t \|_{L_\infty(Q_1^+(\hat{z}))} + \| D u\|_{L_\infty(Q_1^+(\hat{z}))} \leq C(d, \nu,  \alpha) \left( \fint_{Q_2^+(\hat{z})} u^2 \,{\mu} (dz) \right)^{1/2}.
\]
\end{remark}
\noindent
Next, we establish local interior and local boundary Lipschitz estimates for higher order derivatives of solutions to the homogeneous equations.
\begin{lemma} \label{Higher-Lip-est-lemma}
Let  $q\in [1,\infty)$ be a constant. Then under the assumptions of Lemma \ref{loc-L2.est-lemma},
for any $k \in \mathbb{N} \cup \{0\}$, the following statements hold true.
\begin{itemize}
\item[\textup{(i)}] If $\hat{x}_d \geq 2$ and $u \in W^{1,2}_{2}(Q_{2}(\hat{z}))$ is a solution of \eqref{int-Q4-constant-L-0}, then
\begin{equation*} 
\begin{split}
 \|u_{t}\|_{L_\infty(Q_1(\hat{z}))} & +\|u_{tt}\|_{L_\infty(Q_1(\hat{z}))} + \|D u_t\|_{L_\infty(Q_1(\hat{z}))} \leq  C(d, \nu, \alpha, q) \left( \fint_{Q_{2}(\hat{z})} |u_t|^q \,{\mu} (dz) \right)^{1/q},  \\
\|D D_{x'}^k u\|_{L_\infty(Q_{1}(\hat{z}))} & +\|D D_{x'}^k u_t\|_{L_\infty(Q_{1}(\hat{z}))} + \|D^2D_{x'}^k u\|_{L_\infty(Q_{1}(\hat{z}))} \\
 &\leq C(d, \nu, k, \alpha, q) \left( \fint_{Q_{2}(\hat{z})} (|D D_{x'}^ku| +\sqrt \lambda |D_{x'}^ku|)^q \,{\mu} (dz) \right)^{1/q}.
 \end{split}
\end{equation*}
\item[\textup{(ii)}] If $\hat{x}_d =0$ and $u \in W^{1,2}_{2}(Q_{2}^+(\hat{z}))$ is a solution of \eqref{bdr-Q4-constant-L-0}, then
\begin{equation}   \label{H-B-L-infty-u}
\begin{split}
 \|u_{t}\|_{L_\infty(Q_1^{+}(\hat{z}))} & +\|u_{tt}\|_{L_\infty(Q_1^{+}(\hat{z}))} + \|D u_t\|_{L_\infty(Q_1^{+}(\hat{z}))} \leq  C(d, \nu, \alpha, q) \left( \fint_{Q_{2}^+(\hat{z})} |u_t|^q \,{\mu} (dz) \right)^{1/q},  \\
  \|D D_{x'}^k u\|_{L_\infty(Q_{1}^+(\hat{z}))} & +\|D D_{x'}^k u_t\|_{L_\infty(Q_{1}^+(\hat{z}))} + \|D^2D_{x'}^k u\|_{L_\infty(Q_{1}^+(\hat{z}))}  \\
&\leq C(d, \nu, k, \alpha, q) \left( \fint_{Q_{2}^+(\hat{z})} (|D D_{x'}^ku| +\sqrt\lambda |D_{x'}^ku|)^q \,{\mu} (dz) \right)^{1/q}.
\end{split}
\end{equation}
\end{itemize}
\end{lemma}
\begin{proof} We only prove (ii) of the lemma as there is no singularity in (i) and its proof is therefore simpler. Without loss of generality, we can assume $\hat{z} =0$.  Furthermore, by H\"older's inequality for $q>2$ and a standard iteration argument for $q\in [1,2)$ (see, for instance, \cite[p. 75]{FHL}), we only need to consider the case when $q=2$. Observe that the function $u_t$ is still a solution of \eqref{bdr-Q4-constant-L-0}. Therefore, the first estimate of \eqref{H-B-L-infty-u} follows from
Remark \ref{remark-1}. It then remains to prove the second estimate in \eqref{H-B-L-infty-u}. As $D_{x'}^k u$ is again the solution of \eqref{bdr-Q4-constant-L-0}, it is therefore sufficient to prove this estimate with $k =0$. Also, as $u_t$ is a solution of \eqref{bdr-Q4-constant-L-0}, we can apply Lemmas  \ref{Lip-D1.lemma} and \ref{loc-L2.est-lemma} to obtain
\[
 \|Du\|_{L_\infty(Q_{1}^+(\hat{z}))}+ \|u_t\|_{L_\infty(Q_{1}^+(\hat{z}))} +\|D u_t\|_{L_\infty(Q_{1}^+(\hat{z}))} \leq C(d, \nu, \alpha) \left( \fint_{Q_{2}^+(\hat{z})} (|D u|^2 +\lambda u^2) \,{\mu} (dz) \right)^{1/2}.
\]
Hence, it remains to prove that
\[
\|D^2 u\|_{L_\infty(Q_{1}^+(\hat{z}))}  \leq C(d, \nu, \alpha) \left( \fint_{Q_{2}^+(\hat{z})} (|D u|^2 +\lambda u^2) \,{\mu} (dz) \right)^{1/2}.
\]
Recall that $D_{x'}u$ satisfies the same equation as $u$ with the same boundary condition. By using Lemmas \ref{loc-L2.est-lemma} and \ref{Lip-D1.lemma} with some slight modification {as stated in Remark \ref{remark-1}}, we get
\begin{equation}
                                \label{eq11.26}
\|DD_{x'}u\|_{L_\infty(Q^+_{1})} \le C\left(\fint_{Q_2^+} |D_{x'}u|^2 \ \mu(dz)\right)^{\frac 1 2}.
\end{equation}
Now we prove the $L_\infty$-estimate of $D_d^2u$. To this end, we observe that from \eqref{a-dd.cond} and \eqref{div-L-0}, it follows that
\begin{equation}
                                    \label{eq11.45}
D_d(\mu(x_d)  D_d u)=\mu(x_d)\Big[ a_0(x_d) u_t-  \sum_{i=1}^{d-1} \sum_{j=1}^d a_{ij}(x_d) D_{ij} u +\lambda c_0(x_d) u \Big].
\end{equation}
Then, by using \eqref{eq11.45},  (ii) of Lemma \ref{Lip-D1.lemma},  \eqref{eq11.26}, and Lemma \ref{loc-L2.est-lemma}, we get in $Q_{1}^+$
\begin{align*}
|D_d(\mu(x_d)D_d u)|
&\le C\mu(x_d)\sup_{Q^+_{1}}\big(|u_t|+|DD_{x'}u| + \lambda |u|\big)\\
&\le C\mu(x_d)\left(\fint_{Q^+_2}\Big [|Du|^2 +\lambda u^2 \Big] {\ \mu(dz)}\right)^{1/2}.
\end{align*}
This estimate together with the zero boundary condition gives
$$
|\mu(x_d)D_d u|\le C\int_0^{x_d}\mu(s)\,ds\left(\fint_{Q^+_2} \Big[|Du|^2 +\lambda u^2\Big]{\ \mu(dz)}\right)^{1/2}.
$$
Therefore,
\begin{equation*}
\sup_{Q^+_{1}}|D_d u|\le Cx_d\left(\fint_{Q^+_2}\Big[|Du|^2 +\lambda u^2\Big] {\ \mu(dz)}\right)^{1/2}.
\end{equation*}
The last estimate together with \eqref{eq11.45}, (ii) of Lemma \ref{Lip-D1.lemma}, and \eqref{eq11.26} yields
\begin{equation*}
\sup_{Q^+_{1}}|D_{d}^2 u|\le C\left(\int_{Q^+_2}\Big[|Du|^2 +\lambda u^2\Big] {\ \mu(dz)}\right)^{1/2}.
\end{equation*}
The proof of the lemma is then completed.
\end{proof}
Recalling that the notation $(f)_Q$ is defined in \eqref{sharp-def}. As a result of Lemma \ref{Higher-Lip-est-lemma}, we obtain the following mean oscillation estimates for solutions to the homogeneous equations.
\begin{corollary}[Oscillation estimates] \label{h-oss.est-constant} Let $\nu \in (0, 1], \lambda \geq 0, \alpha \in (-1,1)$, $q \in (1,\infty)$ and $\rho >0$ be fixed numbers. Let $\hat{z} = (\hat{t}, \hat{x}', \hat{x}_d) \in \bR \times \bR^{d-1} \times \bR_+$ and $(a_{ij})_{d\times d}: ((\hat{x}_d -6\rho)^+, \hat{x}_d + 6\rho) \rightarrow \bR^{d\times d}$ be a bounded measurable matrix-valued function satisfying \eqref{ellipticity} and \eqref{a-dd.cond}. Assume that \eqref{a-b.zero} holds and $u \in W^{1,2}_{q}(Q_{8\rho}^+(\hat{z}))$ is a solution of
\[
 a_0(x_d) u_t - \mathcal{L}_0 u + \lambda c_0(x_d) u =0\quad \text{in} \quad Q_{6\rho}^+(\hat{z})
\]
with the boundary condition
\[
\lim_{x_d \rightarrow 0^+}x_d^{\alpha}D_d u(t,x', x_d) =0, \quad (t,x') \in Q_{6\rho}'(\hat{z}') \quad \text{if} \quad x_d \leq 6\rho, \quad \text{where} \quad \hat{z'} = (\hat{t}, \hat{x}').
\]
 Then, for every $\kappa \in (0,1)$, it follows that
\[
\begin{split}
& \fint_{Q_{\kappa \rho}^+(\hat{z})} |u - (u)_{Q_{\kappa \rho}^+(\hat{z})}| \ \mu(dz) \leq C(\nu, d, \alpha) \kappa \fint_{Q_{8\rho}^+(\hat{z})} |u| \ \mu(dz), \\
& \fint_{Q_{\kappa \rho}^+(\hat{z})} |u_t - (u_t)_{Q_{\kappa \rho}^+(\hat{z})}| \ \mu(dz) \leq C(\nu, d, \alpha) \kappa  \fint_{Q_{8\rho}^+(\hat{z})} |u_t| \ \mu(dz), \\
& \fint_{Q_{\kappa \rho}^+(\hat{z})} |Du  - (Du)_{Q_{\kappa \rho}^+(\hat{z})}| \ \mu(dz) \leq C(\nu, d, \alpha) \kappa \fint_{Q_{8\rho}^+(\hat{z})} (|Du| +\sqrt \lambda |u|) \ \mu(dz), \\
& \fint_{Q_{\kappa \rho}^+(\hat{z})} |DD_{x'} u - (DD_{x'} u)_{_{Q_{\kappa \rho}^+(\hat{z}})}| \ \mu(dz)  \leq C(\nu, d, \alpha) \kappa  \fint_{Q_{8\rho}^+(\hat{z})} (|D D_{x'} u| +\sqrt \lambda |D_{x'}u| \ \mu(dz).
\end{split}
\]
\end{corollary}
\begin{proof}
Up to a dilation, we may assume without loss of generality that $\rho=1$.  Observe also that since $u_t$ and $D_{x'}u$ satisfy the same equation as $u$, by using the regularity theory established in \cite[Theorem 1.6 and Remark 1.12 (ii)]{Dong-Phan} for weak solutions in divergence form as in \eqref{div-L-0}, we see that $u$, $u_t$, $Du$, $DD_{x'}u$ are in $L_p(Q_{6\rho}^+(\hat{z}))$ for any $p\in (q,\infty)$. Then by Lemma \ref{D-dd-control}, we have $D_d^2 u\in L_p(Q_{6\rho}^+(\hat{z}))$.
Therefore, Lemmas \ref{Lip-est-lemma} and \ref{Higher-Lip-est-lemma} are applicable.  We then consider the following two cases. \\
\noindent
{\bf Case 1:} $\hat{x}_d \geq 2$. In this case $Q_{2}^+(\hat{z}) = Q_2(\hat{z})$. To see the estimate of the oscillation of $u_t$, we use (i) of Lemma \ref{Higher-Lip-est-lemma} with $q=1$ and the doubling property of the $A_2$-weight $\mu$ to obtain
\[
\begin{split}
&\fint_{Q_{\kappa}^+(\hat{z})} |u_t - (u_t)_{Q_{\kappa}^+(\hat{z})}| \ \mu(dz)  \leq \kappa \Big[ \|\partial_t^2u\|_{L_\infty(Q_1(\hat{z}))} + \|D u_t\|_{L_\infty(Q_1(\hat{z}))} \Big] \\
&\quad \leq \kappa C(\nu, d, \alpha)  \fint_{Q_{2}(\hat{z})} |u_t| \ \mu(dz)\leq \kappa C(\nu, d, \alpha) \fint_{Q_{8}^+(\hat{z})} |u_t| \ \mu(dz).
\end{split}
\]
Similarly,
\[
\begin{split}
\fint_{Q_{\kappa}^+(\hat{z})} |DD_{x'} u - (DD_{x'} u)_{_{Q_{\kappa}^+(\hat{z})}}| \ \mu(dz)
& \leq  \kappa \Big[ \|\partial_t D D_{x'}u\|_{L_\infty(Q_1(\hat{z}))} + \|D^2D_{x'}u\|_{L_\infty(Q_1(\hat{z}))} \Big] \\
& \leq \kappa C(\nu, d, \alpha) \fint_{Q_{2}(\hat{z})} (|D D_{x'} u| +\sqrt\lambda |D_{x'}u|) \ \mu(dz)\\
& \leq \kappa C(\nu, d, \alpha)  \fint_{Q_{8}^+(\hat{z})} (|D D_{x'} u| +\sqrt\lambda |D_{x'}u|) \ \mu(dz).
\end{split}
\]
For the oscillation of $Du$, we apply Lemma \ref{Higher-Lip-est-lemma} (i) in a similar way to get
\[
\begin{split}
&\fint_{Q_{\kappa}^+(\hat{z})} |Du - (Du)_{Q_{\kappa}^+(\hat{z})}| \ \mu(dz) \leq \kappa \Big[ \|D^2u\|_{L_\infty(Q_1(\hat{z}))} + \|\partial_tD u\|_{L_\infty(Q_1(\hat{z}))} \Big] \\
&\quad \leq \kappa C(\nu, d, \alpha)  \fint_{Q_{2}(\hat{z})} (|Du| +\sqrt \lambda |u|) \ \mu(dz)\leq \kappa C(\nu, d, \alpha) \fint_{Q_{8}^+(\hat{z})} (|Du| +\sqrt \lambda |u|) \ \mu(dz).
\end{split}
\]
Finally, the estimate of the oscillation of $u$ can be done the same using Remark \ref{remark-1}.\\
\noindent
{\bf Case 2:} $\hat{x}_d < 2$. In this case, let $\bar{z} = (\hat{t}, \hat{x}', 0) \in Q_2(\hat{z})$. Since $\kappa \in (0, 1)$, we observe that
\[
Q_{\kappa}^+(\hat{z}) \subset Q_{3}^+(\bar{z}) \subset Q_{6}^+(\bar{z}) \subset Q_{8}^+(\hat{z}).
\]
Therefore, by Lemma \ref{Higher-Lip-est-lemma} (ii), it follows that
\[
\begin{split}
 &\fint_{Q_{\kappa}^+(\hat{z})} |u_t - (u_t)_{Q_{\kappa}^+(\hat{z})}| \ \mu(dz)  \leq   \kappa \Big[ \|\partial_t^2u\|_{L_\infty(Q_{3}^+(\bar{z}))} + \|D u_t\|_{L_\infty(Q_{3}^+(\bar{z}))} \Big] \\
& \leq \kappa C(\nu, d, \alpha) \fint_{Q_{6}^+(\bar{z})} |u_t| \ \mu(dz)
\leq \kappa C(\nu, d, \alpha) \fint_{Q_{8}^+(\hat{z})} |u_t| \ \mu(dz).
\end{split}
\]
where in the last inequality, we have used the doubling property of the $A_2$-weight $\mu$. The other estimates can be proved exactly in the same way. This completes the proof of the lemma.
\end{proof}

\subsection{$L_p$-estimates for non-homogeneous equations} In this subsection we state and prove two results for the non-homogeneous equation \eqref{constant-L-0}. The first one is a solvability result which particularly proves Theorem \ref{para-main.theorem} when coefficients depend only on $x_d$-variable, $q=p$, and $\omega\equiv 1$.

\begin{theorem}
\label{W-2-p.constant.eqn} Let $\nu \in (0, 1], p \in (1,\infty)$ and $\alpha \in (-1,1)$ be fixed, and let $\mu(s) = |s|^{\alpha}$ for  $s \in \bR \setminus\{0\}$. Suppose that \eqref{ellipticity}, \eqref{a-b.zero}, and \eqref{a-dd.cond} are satisfied. Then, for any $f \in L_p(\Omega_T, d\mu)$ and $\lambda > 0$, there exists a unique strong solution $u \in W^{1,2}_p(\Omega_T, d\mu)$ of the equation \eqref{constant-L-0}. Moreover, the solution $u$ satisfies
\begin{equation} \label{constant-L-p}
\begin{split}
& \norm{u_t}_{L_p(\Omega_T, d\mu)} + \norm{D^2 u}_{L_p(\Omega_T, d\mu)} + \sqrt{\lambda} \norm{Du}_{L_p(\Omega_T, d\mu)} + \lambda \norm{u}_{L_p(\Omega_T, d\mu)} \\
& \leq N(d, \nu,\alpha, p)  \norm{f}_{L_p(\bR^d_+, d\mu)}.
\end{split}
\end{equation}
\end{theorem}
\noindent
We note that even though the coefficient matrix $(a_{ij})_{d\times d}$ only depends in $x_d$-variable, Theorem \ref{W-2-p.constant.eqn}  is new and also very important in applications. 
In particular, the theorem will be useful in the next section when we study parabolic equations \eqref{main-eqn} with partially weighted \textup{VMO} coefficients. Our first lemma is a global $L_2$-estimate lemma, which in turn  gives Theorem \ref{W-2-p.constant.eqn} when $p =2$.

\begin{lemma}[Global $L_2$-estimates] \label{gbl-L2-constant.est}
Under the assumptions of Theorem \ref{W-2-p.constant.eqn}, for any $f \in L_2(\Omega_T, d\mu)$ and $\lambda > 0$ there exists unique solution $u \in W^{1,2}_2(\Omega_T, d\mu)$ of  \eqref{constant-L-0}. Moreover,
\begin{equation} \label{L2-constant-est.1}
\begin{split}
& \norm{u_t}_{L_2(\Omega_T, d\mu)} + \norm{D^2 u}_{L_2(\Omega_T, d\mu)} + \sqrt{\lambda} \norm{Du}_{L_2(\Omega_T, d\mu)} + \lambda \norm{u}_{L_2(\Omega_T, d\mu)} \\
& \leq N(d, \nu,\alpha)\norm{f}_{L_2(\Omega_T, d\mu)}.
\end{split}
\end{equation}
\end{lemma}
\begin{proof} We first prove the estimate \eqref{L2-constant-est.1} for each solution $u \in W^{1,2}_2(\Omega_T, d\mu)$ of  \eqref{constant-L-0}. Since the coefficients are independent of $t$ and $x'$, by mollifying the equation in $x'$ and $t$, we may assume that $u_t,D_{x'}u\in W^{1,2}_2(\Omega_T, d\mu)$.
By multiplying the equation \eqref{div-L-0} by  $\lambda u$ and integrating in $\Omega_T$, and then using the integration by parts, the ellipticity condition \eqref{ellipticity}, and the condition \eqref{a-b.zero}, we get the energy inequality
\[
\begin{split}
 \lambda \nu \int_{\Omega_T} \mu(x_d) |Du|^2 \ dxdt + \lambda^2 \nu \int_{\Omega_T} \mu(x_d) |u|^2 \ dxdt  & \leq \lambda \int_{\Omega_T}\mu(x_d) |f(t,x)| |u(t,x)|  \ dxdt.
\end{split}
\]
Then using Young's inequality for the term on the right hand side of the above estimate, we obtain
\begin{equation}                               \label{eq3.27}
\begin{split}
 \lambda \int_{\Omega_T} \mu(x_d) |Du|^2 \ dxdt + \lambda^2 \int_{\Omega_T} \mu(x_d) |u|^2 \ dxdt  & \leq N(\nu)  \int_{\Omega_T}\mu(x_d) f^2(t,x) \ dxdt.
\end{split}
\end{equation}
Now, we multiply the equation \eqref{div-L-0} with $D_{kk} u$ for $k \in \{1,2,\ldots, d-1\}$. Because $D_{k} u$ satisfies the same equation with the same boundary condition as $u$, we can use the integration by parts to get
\[
\int_{\Omega_T} \mu(x_d) a_{ij}(x_d) D_{jk} u D_{ik} u \ dxdt + \lambda \int_{\Omega_T} \mu(x_d) c_0(x_d)|D_k u|^2 \ dxdt  \le - \int_{\Omega_T} \mu(x_d) f D_{kk} u \ dxdt.
\]
Then, by using the ellipticity condition \eqref{ellipticity} and \eqref{a-b.zero}, H\"{o}lder's inequality, and Young's inequality we obtain
\begin{equation} \label{eq3.27-p}
\int_{\Omega_T}\mu(x_d) |D D_{x'} u|^2 \ dx dt + \lambda \int_{\Omega_T} \mu(x_d) |D_{x'} u|^2 \ dxdt \leq N(d,\nu) \int_{\Omega_T}\mu(x_d) f(t,x)^2 \ dxdt
\end{equation}
Next, we estimate the weighted $L_2$ norm of $u_t$. Recall that $a_{dd}=1$. We rewrite the first equation of \eqref{div-L-0} into
\begin{equation}
                                \label{eq7.48}
x_d^\alpha a_0(x_d) u_t- D_d(x_d^\alpha D_d u)=x_d^\alpha \tilde f,
\end{equation}
where
$$
\tilde f=f+ \sum_{i=1}^{d-1}\sum_{j=1}^{d}a_{ij}D_{ij} u -  \lambda c_0 u.
$$
We test \eqref{eq7.48} by $u_t$ and integrate in $\Omega_T$, and integrate by parts using the zero boundary condition to get
$$
\int_{\Omega_T}\mu(x_d) a_0(x_d) u_t^2 \ dxdt +\int_{\Omega_T} \mu(x_d)  D_d u D_d u_t \ dx dt= \int_{\Omega_T} \mu(x_d) \tilde f(t,x) u_t(t,x) \ dxdt.
$$
Since the second term on the left-hand side above is nonnegative, by Young's inequality, \eqref{a-b.zero}, \eqref{eq3.27}, and \eqref{eq3.27-p} we obtain
\begin{equation*}
\int_{\Omega_T}\mu(x_d)  u_t^2 \ dxdt\leq N(d,\nu) \int_{\Omega_T}\mu(x_d) f^2(t,x) \ dxdt.
\end{equation*}
Finally, the estimate of $D_{d}^2 u$ follows from Lemma \ref{D-dd-control} with $\omega_0, \omega_1, \omega_2 \equiv1$.

Next, we prove the unique solvability of \eqref{constant-L-0}. As the equation \eqref{constant-L-0} can be written in the divergence form \eqref{div-L-0}, by \cite[Theorem 1.10]{Dong-Phan},  there is a unique weak solution
$u$ of \eqref{div-L-0} such that $u,Du\in L_2(\Omega_T,d\mu)$. Observe that in this case, we do not require any regularity condition on the coefficient $(a_{ij})_{i, j=1}^d$. In fact, in \cite[Theorem 1.10]{Dong-Phan} $\lambda$ is assumed to be sufficiently large. However, in our case by a simple scaling argument, we only need $\lambda>0$.
By mollifying the equation in $x'$ and $t$, we may assume that $u_t^{(\varepsilon)}, D_{x'} u^{(\varepsilon)},DD_{x'} u^{(\varepsilon)}\in L_2(\Omega_T,d\mu)$. It follows from Lemma \ref{D-dd-control} that $u^{(\varepsilon)} \in W^{1,2}_2(\Omega_T, d\mu)$ is a strong solution of \eqref{constant-L-0} with $f^{(\varepsilon)}$ in place of $f$. By the proof above, we have \eqref{L2-constant-est.1} with $u^{(\varepsilon)}$ and $f^{(\varepsilon)}$ in place of $u$ and $f$. Now taking the limit as $\varepsilon\to 0$, we get \eqref{L2-constant-est.1}. The uniqueness follows from \eqref{L2-constant-est.1}. The proof of the lemma is then completed.
\end{proof}
Next, we establish the oscillation estimates for $u$, $u_t$, $Du$, and $DD_{x'}u$ for the non-homogeneous equation \eqref{constant-L-0}.
\begin{proposition}[Oscillation estimates] \label{Osc-pro-constant}
 Under the assumptions of Theorem \ref{W-2-p.constant.eqn}, assume that $f \in L_{2, \textup{loc}}(\Omega_T, d\mu)$ and $u \in W^{1,2}_{2, \textup{loc}}(\Omega_T, d\mu)$ is a solution of the equation \eqref{constant-L-0}. Then, for any $\hat{z} = (\hat{t}, \hat{x}', \hat{x}_d) \in \overline{\Omega}_T$, $\lambda > 0$, and for $\kappa \in (0,1)$
\begin{equation}
                                \label{eq10.19}
\begin{split}
& \fint_{Q_{\kappa\rho}^+(\hat{z})} |\hat{U}- (\hat{U})_{Q_{\kappa \rho}^+(\hat{z})}| \ \mu(dz)\\
& \leq C(\nu, d, \alpha)\left[\kappa  \fint_{Q_{8\rho}^+(\hat{z})} |\hat{U}| \ \mu(dz) + \kappa^{-(d+3)/2}\left( \fint_{Q_{8\rho}^+(\hat{z})} |f(t,x)|^2 \ \mu(dz) \right)^{1/2}\right],
\end{split}
\end{equation}
where $\hat{U} = (\lambda u, u_t, \sqrt{\lambda} Du, D D_{x'} u)$ and $(\hat{U})_{Q_{\kappa \rho}^+(\hat{z})}$ is defined as in \eqref{sharp-def}.
\end{proposition}
\begin{proof} Let $v \in W^{1,2}_2(\Omega_T)$ be the solution of
\[
\left\{
\begin{array}{ccc}
a_0(x_d)v_t - \mathcal{L}_0 v(t, x) + \lambda c_0(x_d)u& = & f(t,x) \mathbf{1}_{Q_{8\rho}^+(\hat{z})}(t,x)  \\
\displaystyle{\lim_{x_d \rightarrow 0^+} x_d^\alpha   D_d v(t, x', x_d)} & = & 0
\end{array} \quad \text{in} \quad  \Omega_T, \right.
\]
where $\mathbf{1}_{Q_{8\rho}^+(\hat{z})}$ denotes the characteristic function of the cylinder $Q_{8\rho}^+(\hat{z})$. Observe that the existence of $v$ is  ensured by Lemma \ref{gbl-L2-constant.est}. Moreover, it follows from this lemma that
\begin{equation*} 
\begin{split}
& \|v_t\|_{L_2(\Omega_T, d\mu)} +  \|D^2 v\|_{L_2(\Omega_T, d\mu)} + \sqrt{\lambda} \| Dv\|_{L_2(\Omega_T, d\mu)} + \lambda \| v\|_{L_2(\Omega_T, d\mu)}\\
& \leq C(\nu, d, \alpha) \|f\|_{L_{2}(Q_{8\rho}^+(\hat{z}), d\mu)}.
\end{split}
\end{equation*}
This estimate and the doubling property of the $A_2$-weight $\mu$ particularly imply that
\begin{equation} \label{const-hat-v.est}
\begin{split}
& \left(\fint_{Q_{\kappa \rho}^+(\hat{z})} |\Hat{V}|^2 \ \mu(dz)\right)^{1/2} \leq \frac{C(\nu, d, \alpha)}{\kappa^{(d+3)/2}} \left(\fint_{Q_{8\rho}^+(\hat{z})} |f|^2 \ \mu(dz)\right)^{1/2}, \ \text{and} \\
& \left(\fint_{Q_{8\rho}^+(\hat{z})} |\Hat{V}|^2 \ \mu(dz)\right)^{1/2} \leq C(\nu, d, \alpha) \left(\fint_{Q_{8\rho}^+(\hat{z})} |f|^2 \ \mu(dz)\right)^{1/2},
\end{split}
\end{equation}
where $\hat{V} = (\lambda v, v_t, \sqrt{\lambda} Dv, D D_{x'} v)$.  Now, let $w = u -v \in W^{1,2}_2(Q_{8\rho}^+(\hat{z}))$, which satisfies
\[
a_0(x_d)w_t - \mathcal{L}_0 w + \lambda  c_0(x_d)w = 0 \quad \text{in} \quad Q_{6\rho}^+(\hat{z}).
\]
Moreover, if $\hat{x}_d \leq 6\rho$, the solution $w$ also satisfies the boundary condition
\[
\displaystyle{\lim_{x_d \rightarrow 0^+} x_d^\alpha  D_d w(t, x', x_d)} =  0,   \quad (t, x') \in Q_{6\rho}'(\hat{z}'), \quad \text{with} \quad \hat{z}' = (\hat{t}, \hat{x}').
\]
Hence, it follows from Corollary \ref{h-oss.est-constant} that
\begin{equation} \label{constant-hat-w.est}
\fint_{Q_{\kappa \rho}^+(\hat{z})} |\hat{W} - (\hat{W})_{Q_{\kappa \rho}^+(\hat{z})}| \ \mu(dz)  \leq C(\nu, d, \alpha) \kappa \fint_{Q_{8\rho}^+(\hat{z})} |\hat{W}| \ \mu(dz),
\end{equation}
where $\hat{W} = (\lambda w, w_t, \sqrt{\lambda} Dw, D D_{x'} w)$. Now, to estimate  the oscillation of $\hat{U}$, we note that
\[
\fint_{Q_{\kappa \rho}^+(\hat{z}} |\hat{U} - (\hat{U})_{Q_{\kappa \rho}^+(\hat{z})}| \ \mu(dz) \leq 2 \fint_{Q_{\kappa \rho}^+(\hat{z}} |\hat{U} - c| \ \mu(dz), \quad \text{for every} \ c \in \bR.
\]
Then, by taking $c =(\hat{W})_{Q_{\kappa \rho}^+(\hat{z})}$, and using the triangle inequality and H\"{o}lder's inequality, we see that
\[
\begin{split}
\fint_{Q_{\kappa \rho}^+(\hat{z})} |\hat{U} - (\hat{U})_{Q_{\kappa \rho}^+(\hat{z})}| \ \mu(dz) & \leq 2 \fint_{Q_{\kappa \rho}^+(\hat{z})} |\hat{U} - (\hat{W})_{Q_{\kappa \rho}}^+(\hat{z})| \ \mu(dz) \\
& \leq 2 \left[ \fint_{Q_{\kappa \rho}^+(\hat{z})} |\hat{W}  - (\hat{W})_{Q_{\kappa \rho}^+(\hat{z})}| \ \mu(dz) + \left(  \fint_{Q_{\kappa \rho}^+(\hat{z})} |\hat{V}|^2 \ \mu(dz) \right)^{1/2} \right].
\end{split}
\]
From this estimate, the first estimate in \eqref{const-hat-v.est}, and  \eqref{constant-hat-w.est}, we see that
\[
\begin{split}
& \fint_{Q_{\kappa \rho}^+(\hat{z})} |\hat{U} - (\hat{U})_{Q_{\kappa \rho}^+(\hat{z})}| \ \mu(dz) \\
& \leq C \left[\kappa \fint_{Q_{8\rho}^+(\hat{z})}|\hat{W}(z)| \ \mu(dz) + \kappa^{-\frac{d+3}{2}} \left(\fint_{Q_{8\rho}^+(\hat{z})} |f(z)|^2 \ \mu(dz) \right)^{\frac 1 2} \right] \\
& \leq  C \left[\kappa \fint_{Q_{8\rho}^+(\hat{z})}|\hat{U} (z)| \ \mu(dz) + \kappa \left(\fint_{Q_{8\rho}^+(\hat{z})}|\hat{V} (z)|^2 \ \mu(dz)\right)^{\frac 1 2} + \kappa^{-\frac{d+3}{2}} \left(\fint_{Q_{\rho}^+(\hat{z})} |f(z)|^2 \ \mu(dz) \right)^{\frac 1 2} \right] .
\end{split}
\]
Now, using the second estimate in \eqref{const-hat-v.est}, we can control the middle term on the right hand side of the last estimate and infer \eqref{eq10.19}.
The proof of the lemma is therefore completed.
\end{proof}

Now we can prove Theorem \ref{W-2-p.constant.eqn}.
\begin{proof}[Proof of Theorem \ref{W-2-p.constant.eqn}]
Note that the case $p =2$ is proved in Lemma \ref{gbl-L2-constant.est}. It then remains to consider the case $p\not=2$. We split the proof into two cases:\\
\noindent
{\bf Case 1:} $p > 2$. We first prove the estimate \eqref{constant-L-p}.  Let $u \in W^{1,2}_p(\Omega_T, d\mu)$ be a solution of \eqref{constant-L-0}.  By applying Proposition \ref{Osc-pro-constant}, we  can control the sharp function of $\hat{U}$ by
\[
\hat{U}{^\#}(z) \leq C(\nu, d, \alpha) \left[ \kappa  \M(|\hat{U}|)(z) + \kappa^{-\frac{d+3}{2}} \M(|f|^2)(z)^{1/2} \right] \quad \text{for any} \ z \in \Omega_T,
\]
where $\kappa \in (0,1)$. Then, by using the Fefferman-Stein theorem for sharp functions and Hardy-Littlewood maximal function theorem (cf. \eqref{Maximal-L-p}) we obtain
\[
\begin{split}
\norm{\hat{U}}_{L_p(\Omega_T, d\mu)} & \leq C(d, \alpha, p) \norm{ \hat U^{\#}}_{L_p(\Omega_T, d\mu)}\\
&\leq  C(\nu, d, \alpha,p) \left[\kappa  \norm{\M(|\hat{U}|)}_{L_p(\Omega_T, d\mu)} + \kappa^{-\frac{d+3}{2}} \norm{\M(|f|^2)^{1/2}}_{L_p(\Omega_T, d\mu)} \right] \\
& \leq  C(\nu, d, \alpha, p) \left[ \kappa \norm{\hat{U}}_{L_p(\Omega_T, d\mu)} +\kappa^{-\frac{d+3}{2}} \norm{f}_{L_p(\Omega_T, d\mu)} \right].
\end{split}
\]
By choosing $\kappa$ sufficiently small depending only on $d, \nu, \alpha$, and $p$, we obtain
\begin{equation*}
\norm{\hat{U}}_{L_p(\Omega_T, d\mu)} \leq C(d, \nu, \alpha, p)  \norm{f}_{L_p(\Omega_T,  d\mu)}.
\end{equation*}
This and the definition of $\hat{U}$ imply that
\begin{align*} 
&\|u_t\|_{L_p(\Omega_T, d\mu)} + \|DD_{x'} u\|_{L_p(\Omega_T, d\mu)} + \sqrt{\lambda} \|Du\|_{L_p(\Omega_T, d\mu)} + \lambda \|u\|_{L_p(\Omega_T, d\mu)}\nonumber\\
&\leq C(d, \nu, \alpha, p) \|f\|_{L_p(\Omega_T, d\mu)},
\end{align*}
which together with Lemma \ref{D-dd-control} complete the proof of \eqref{constant-L-p}.

Finally, the existence and uniqueness of solutions can be proved in exactly the same way as in Lemma \ref{gbl-L2-constant.est}. 
\\
\noindent
{\bf Case 2}: $p \in (1,2)$. In this case, we consider the equation in divergence form as in  \eqref{div-L-0}. Using \cite[Theorem 1.10]{Dong-Phan}, we obtain
\begin{equation} \label{10.9-1.est}
\sqrt{\lambda} \norm{Du}_{L_p(\Omega_T, d\mu)} + \lambda \norm{u}_{L_p(\Omega_T, d\mu)} \leq C(d, \nu, p, \alpha) \norm{f}_{L_p(\Omega_T, d\mu)}.
\end{equation}
Then, with the help of the difference quotient, we can formally  differentiate the equation, and see that $D_{x'}u$ is also a solution of the same equation  \eqref{div-L-0} with $f$ replaced by $D_{x'} f$. Therefore,  using \cite[Theorem 1.10]{Dong-Phan} again, we see that
\begin{equation} \label{10.9-2.est}
 \norm{DD_{x'}u}_{L_p(\Omega_T, d\mu)}  \leq C(d, \nu, p, \alpha) \norm{f}_{L_p(\Omega_T, d\mu)}.
\end{equation}
Now, for each fixed $x' \in \bR^{d-1}$, we consider $u$ is a function of $(t,x_d)$-variable, and write  equation \eqref{constant-L-0} as
\[
 a_0(x_d)x_d^\alpha u_t - D_d(x_d^\alpha D_d u)+\lambda c_0(x_d) u  = x_d^\alpha F  \quad \text{in} \quad  \hat{\Omega}_T,
\]
where $\hat{\Omega}_T = (-\infty, T] \times (0, \infty)$ and
\[
F(t, x_d) = f(t, x', x_d) - \sum_{(i, j) \not=(d,d)} a_{ij} D_{ij}u .
\]
Now we use a duality argument. Let $p'=p/(p-1)\in (2,\infty)$. For any $g\in C_0^\infty(\hat{\Omega}_T)$, let $v\in W^{1,2}_{p'}(\bR\times \bR^+)$ be the unique solution to
\begin{equation} \label{eq11.08}
 \left\{
 \begin{array}{ccl}
-a_0(x_d)x_d^\alpha v_t - D_d(x_d^\alpha D_d v)+\lambda c_0(x_d) v  &=& x_d^\alpha g \mathbf{1}_{(-\infty,T)}(t) \\
 \displaystyle{\lim_{x_d \rightarrow 0^+}\displaystyle{x_d^\alpha D_d v}} & =& 0
 \end{array} \quad \text{in} \quad \bR\times \bR^+,
 \right.
\end{equation}
which satisfies
\begin{equation}
                                    \label{eq4.25}
\|v_t\|_{L_{p'}(\bR\times \bR^+, d\mu)}\le C(\nu,p,\alpha)\|g\|_{L_{p'}(\hat\Omega_T, d\mu)}.
\end{equation}
The existence of such solution and \eqref{eq4.25} follow from {\bf Case I} with a change of variable $t\to -t$. Since $g 1_{(-\infty,T)}(t)=0$ for $t\ge T$, it is easily seen that $v=0$ for $t\ge T$. Since $g$ is smooth and supported on $t\in (\infty, T)$, by using the technique of finite difference quotients, we see that $v_t\in W^{1,2}_{p'}(\bR\times \bR^+)$ satisfies \eqref{eq11.08} with $g_t$ in place of $g$.
Using integration by parts and the boundary conditions of $u$ and $v$, we have
\begin{align*}
&\int_{\hat{\Omega}_T} u_t(t,x',x_d)x_d^\alpha g\,dx_d dt\\
&=\int_{\hat{\Omega}_T} u_t(t,x',x_d)\Big[-a_0(x_d)x_d^\alpha v_t - D_d(x_d^\alpha D_d v)+\lambda c_0(x_d) v\Big]\,dx_d dt\\
&=\int_{\hat{\Omega}_T} {\Big[}-a_0(x_d) x_d^\alpha u_t(t,x',x_d)v_t +u(t,x',x_d) \big(D_d(x_d^\alpha D_d v_t)-\lambda c_0(x_d) v_t\big){\Big]}\,{dx_d dt} \\
&=\int_{\hat{\Omega}_T} {\Big[}-a_0(x_d) x_d^\alpha u_t(t,x',x_d)v_t -D_d u(t,x',x_d) x_d^\alpha D_d v_t-\lambda c_0(x_d) u(t,x',x_d) v_t{\big]\, dx_ddt}\\
&=\int_{\hat{\Omega}_T} -x_d^\alpha F v_t \,dx_d dt.
\end{align*}
It then follows from \eqref{eq4.25} that
\[
\begin{split}
\left|\int_{\hat{\Omega}_T} u_t(t,x',x_d) g\,\mu(dx_d)dt\right| & \le \|F\|_{L_p(\hat \Omega_T, {d}\mu)}\|v_t\|_{L_{p'}(\hat \Omega_T, {d}\mu)}
\\
& \le C(\nu,p,\alpha)\|F\|_{L_p(\hat \Omega_T, {d}\mu)}\|g\|_{L_{p'}(\hat \Omega_T, {d}\mu)}.
\end{split}
\]
Since $g\in C_0^\infty(\hat{\Omega}_T)$ is arbitrary, we obtain
\[
\norm{u_t(\cdot, x', \cdot)}_{L_p(\hat{\Omega}_T, {d} \mu)} \leq C(\nu, p, \alpha) \norm{F(\cdot, x', \cdot)}_{L_p(\hat{\Omega}_T, {d} \mu)}.
\]
Then, integrating this last estimate with respect to $x' \in \bR^{d-1}$, we obtain
\[
\norm{u_t}_{L_p(\Omega_T, d\mu)} \leq C(\nu, p, \alpha) \norm{F}_{L_p(\Omega_T, d\mu)}.
\]
From this, \eqref{10.9-1.est}, and \eqref{10.9-2.est}, we infer that
\[
\begin{split}
& \norm{u_t}_{L_p(\Omega_T, d\mu)} +  \norm{DD_{x'}u}_{L_p(\Omega_T, d\mu)}  +  \sqrt{\lambda} \norm{Du}_{L_p(\Omega_T, d\mu)} + \lambda \norm{u}_{L_p(\Omega_T, d\mu)} \\
& \leq C(d, \nu, p, \alpha) \norm{f}_{L_p(\Omega_T, d\mu)},
\end{split}
\]
which together with Lemma \ref{D-dd-control}  implies \eqref{constant-L-p}. As before, the existence and uniqueness of solutions can be proved in the same way as in Lemma \ref{gbl-L2-constant.est}. The proof of the theorem is now completed.
 \end{proof}

Next, we prove the following corollary of Theorem \ref{W-2-p.constant.eqn} giving the mean oscillation estimates of solution of  \eqref{constant-L-0}. The result is an improved version of Proposition \ref{Osc-pro-constant} and it is needed in the next section.
 \begin{corollary} \label{Osc-L-q} Let $\nu \in (0, 1]$, $\alpha \in (-1,1)$, and $q \in (1, \infty)$ be fixed. Let $\lambda>0, \rho>0$,  $\hat{z} = (\hat{t}, \hat{x}', \hat{x}_d) \in \overline{\Omega}_T$,  and $\mu(s) = |s|^{\alpha}$ for  $s \in \bR \setminus\{0\}$. Assume that $(a_{ij})_{d\times d}: \bR_+ \rightarrow \bR^{d \times d}$ is a bounded measurable matrix-valued function satisfying \eqref{ellipticity} and \eqref{a-dd.cond}. Assume that \eqref{a-b.zero} holds, $f \in L_{q}(Q_{8\rho}^+(\hat{z}), d\mu)$, and $u \in W^{1,2}_{q}(Q_{8\rho}^+(\hat{z}), d\mu)$ is a solution of the equation
 \[
 a_0(x_d) u_t - \mathcal{L}_0 u + \lambda c_0(x_d) u =f \quad \text{in} \quad Q_{6\rho}^+(\hat{z})
 \]
 and if $\hat{x}_d \leq 6\rho$, $u$ satisfies the boundary condition
 \[
 \lim_{x_d \rightarrow 0^+}x_d^\alpha D_{d}u (t,x', x_d)=0 \quad (t,x') \in Q_{6\rho}'(\hat{z}') \quad \text{for} \quad \hat{z}' =(\hat{t}, \hat{x}').
 \]
 Then, for every $\kappa \in (0,1)$
\begin{equation}
                                \label{eq10.19-q}
\begin{split}
& \fint_{Q_{\kappa\rho}^+(\hat{z})} |\hat{U}- (\hat{U})_{Q_{\kappa \rho}^+(\hat{z})}| \ \mu(dz)\\
& \leq C(\nu, d, \alpha, q)\left[\kappa  \fint_{Q_{8\rho}^+(\hat{z})} |\hat{U}| \ \mu(dz) + \kappa^{-(d+3)/q}\left( \fint_{Q_{8\rho}^+(\hat{z})} |f(t,x)|^q \ \mu(dz) \right)^{1/q}\right],
\end{split}
\end{equation}
where $\hat{U} = (\lambda u, u_t, \sqrt{\lambda} Du, D D_{x'} u)$ and $(\hat{U})_{Q_{\kappa \rho}^+(\hat{z})}$ is defined as in \eqref{sharp-def}.
\end{corollary}
\begin{proof} The proof is similar to that of Proposition \ref{Osc-pro-constant}. However, instead of using the $L_2$-estimates in Lemma \ref{gbl-L2-constant.est}, we use Theorem \ref{W-2-p.constant.eqn}. We provide the details of the proof for completeness. By Theorem \ref{W-2-p.constant.eqn}, there is a unique solution $v \in W^{1,2}_q(\Omega_T)$ to the equation
\[
\left\{
\begin{array}{cccl}
a_0(x_d) v_t - \mathcal{L}_0 v(t, x) + \lambda c_0(x_d) u& = & f(t,x) \mathbf{1}_{Q_{8\rho}^+(\hat{z})}(t,x) \\
\displaystyle{\lim_{x_d \rightarrow 0^+} x_d^\alpha D_d v(t, x', x_d)} & = & 0
\end{array}  \quad \text{in} \quad \Omega_T, \right.
\]
where $\mathbf{1}_{Q_{8\rho}^+(\hat{z})}$ denotes the characteristic function of the cylinder $Q_{8\rho}^+(\hat{z})$.
Moreover, 
\begin{equation*} 
\begin{split}
& \|v_t\|_{L_q(\Omega_T, d\mu)} +  \| D^2 v\|_{L_q(\Omega_T, d\mu)} + \sqrt{\lambda} \| Dv\|_{L_q(\Omega_T, d\mu)} + \lambda \| v\|_{L_q(\Omega_T, d\mu)}\\
& \leq C(\nu, d, \alpha, q) \|f\|_{L_{q}(Q_{8\rho}^+(\hat{z}), d\mu)}.
\end{split}
\end{equation*}
This estimate and the doubling property of the $A_2$-weight $\mu$ imply that
\begin{equation} \label{const-hat-v.est-q}
\begin{split}
& \left(\fint_{Q_{\kappa \rho}^+(\hat{z})} |\Hat{V}|^q \ \mu(dz)\right)^{1/q} \leq \frac{C(\nu, d, \alpha, q)}{\kappa^{(d+3)/q}} \left(\fint_{Q_{8\rho}^+(\hat{z})} |f|^q \ \mu(dz)\right)^{1/q}, \ \text{and} \\
& \left(\fint_{Q_{8\rho}^+(\hat{z})} |\Hat{V}|^q \ \mu(dz)\right)^{1/q} \leq C(\nu, d, \alpha,q) \left(\fint_{Q_{8\rho}^+(\hat{z})} |f|^q \ \mu(dz)\right)^{1/q},
\end{split}
\end{equation}
where $\hat{V} = (\lambda v, v_t, \sqrt{\lambda} Dv, D D_{x'} v)$.  Next, let $w = u -v\in W^{1,2}_q(Q_{8\rho}^+(\hat{z}))$,  which satisfies
\[
a_0(x_d) w_t - \mathcal{L}_0 w + \lambda  c_0(x_d) w = 0 \quad \text{in} \quad Q_{6\rho}^+(\hat{z}).
\]
In addition, if $\hat{x}_d \leq 6\rho$, $w$ satisfies the following boundary condition
\[
\displaystyle{\lim_{x_d \rightarrow 0^+} x_d^\alpha  D_d w(t, x', x_d)} =  0,   \quad (t, x') \in Q_{6\rho}'(\hat{z}'), \quad \text{with} \quad \hat{z}' = (\hat{t}, \hat{x}').
\]
Then it follows from Corollary \ref{h-oss.est-constant} that
\begin{equation} \label{constant-hat-w.est-q}
\fint_{Q_{\kappa \rho}^+(\hat{z})} |\hat{W} - (\hat{W})_{Q_{\kappa \rho}^+(\hat{z})}| \ \mu(dz)  \leq C(\nu, d, \alpha) \kappa \fint_{Q_{8\rho}^+(\hat{z})} |\hat{W}| \ \mu(dz),
\end{equation}
where $\hat{W} = (\lambda w, w_t, \sqrt{\lambda} Dw, D D_{x'} w)$.  To control the the oscillation of $\hat{U}$, we  recall that
\[
\fint_{Q_{\kappa \rho}^+(\hat{z})} |\hat{U} - (\hat{U})_{Q_{\kappa \rho}^+(\hat{z})}| \ \mu(dz) \leq 2 \fint_{Q_{\kappa \rho}^+(\hat{z})} |\hat{U} - c| \ \mu(dz), \quad \text{for every} \ c \in \bR.
\]
Consequently, by taking $c =(\hat{W})_{Q_{\kappa \rho}^+(\hat{z})}$, and using the triangle inequality and H\"{o}lder's inequality, we obtain
\[
\begin{split}
\fint_{Q_{\kappa \rho}^+(\hat{z})} |\hat{U} - (\hat{U})_{Q_{\kappa \rho}^+(\hat{z})}| \ \mu(dz) & \leq 2 \fint_{Q_{\kappa \rho}^+(\hat{z})} |\hat{U} - (\hat{W})_{Q_{\kappa \rho}}^+(\hat{z})| \ \mu(dz) \\
& \leq 2 \left[ \fint_{Q_{\kappa \rho}^+(\hat{z})} |\hat{W}  - (\hat{W})_{Q_{\kappa \rho}^+(\hat{z})}| \ \mu(dz) + \left(  \fint_{Q_{\kappa \rho}^+(\hat{z})} |\hat{V}|^q \ \mu(dz) \right)^{1/q} \right].
\end{split}
\]
From this estimate, the first estimate in \eqref{const-hat-v.est-q}, and  \eqref{constant-hat-w.est-q}, we see that
\[
\begin{split}
& \fint_{Q_{\kappa \rho}^+(\hat{z})} |\hat{U} - (\hat{U})_{Q_{\kappa \rho}^+(\hat{z})}| \ \mu(dz) \\
& \leq C \left[\kappa \fint_{Q_{8\rho}^+(\hat{z})}|\hat{W}(z)| \ \mu(dz) + \kappa^{-\frac{d+3}{q}} \left(\fint_{Q_{8\rho}^+(\hat{z})} |f(z)|^q \ \mu(dz) \right)^{
\frac 1q} \right] \\
& \leq  C \left[\kappa \fint_{Q_{8\rho}^+(\hat{z})}|\hat{U} (z)| \ \mu(dz) + \kappa \left(\fint_{Q_{8\rho}^+(\hat{z})}|\hat{V} (z)|^q \ \mu(dz)\right)^{\frac 1q} + \kappa^{-\frac{d+3}{q}} \left(\fint_{Q_{\rho}^+(\hat{z})} |f(z)|^q \ \mu(dz) \right)^{\frac 1q} \right] .
\end{split}
\]
Finally, using the second estimate in \eqref{const-hat-v.est-q}, we can control the middle term on the right hand side of the last estimate and infer \eqref{eq10.19-q}.
The proof is completed.
\end{proof}
\section{Equations with singular partially weighted \textup{BMO} coefficients} \label{VMO-section}

This section is devoted to the proofs of Theorems \ref{para-main.theorem} and \ref{elli-main.theorem}.  Recall that $\Omega_T = (-\infty, T] \times \bR^d_{+}$, $(a_{ij})_{d \times d}: \Omega_T \rightarrow \bR^{d\times d}$ is a bounded, measurable matrix-valued function satisfying the ellipticity condition \eqref{ellipticity} and \eqref{extension-type-matrix}, and $a_0, c_0: \Omega_T \rightarrow \bR$ are measurable functions satisfying \eqref{a-b.zero0}. We first focus our attention on the equation \eqref{main-eqn} which is the parabolic equation in non-divergence form with singular coefficients:
\begin{equation} \label{eqn.variable-coeff}
\left\{
\begin{array}{ccc}
a_0u_t - \mathcal{L} u(t,x)  + \lambda c_0 u&  = & f(t,x) \\
\displaystyle{\lim_{x_d \rightarrow 0^+} x_d^\alpha D_d u} & = & 0
\end{array} \quad \text{in} \quad \Omega_T, \right.
\end{equation}
where
\[
\mathcal{L} u(t,x) = a_{ij}(t,x)D_{ij} u(t,x) + \frac{\alpha}{x_d}  a_{dd}(t,x) D_{d}u(t,x).
\]
\noindent
We first state and prove a lemma about the oscillation estimates for the solutions.
\begin{lemma} \label{lemma.3.3} Let $\nu \in (0,1)$, $q \in (1, \infty)$, $\alpha \in (-1,1)$, $p \in (q, \infty)$ and assume that \eqref{ellipticity}, \eqref{a-b.zero0}, and \eqref{extension-type-matrix} hold. Let $\lambda>0$ and $\rho, \rho_1, \rho_0 \in (0,1)$,  $\hat{z} = (\hat{t}, \hat{x}', \hat{x}_d), z_0 \in \overline{\Omega}_T$, $t_1 \in \bR$ and $f \in L_{q}(Q_{8\rho}^+(\hat{z}), d\mu)$. Assume that $u \in W^{1,2}_{p}(Q_{8\rho}^+(\hat{z}), d\mu)$ vanishing outside $(t_1 -(\rho_0\rho_1)^2, t_1]$ is a solution of the equation
 \[
 u_t - \mathcal{L} u + \lambda u =f \quad \text{in} \quad Q_{6\rho}^+(\hat{z}),
 \]
 and if $\hat{x}_d \leq 6\rho$, $u$ satisfies the boundary condition
 \[
 \lim_{x_d \rightarrow 0^+}x_d^\alpha D_{d}u(t,x', x_d)=0, \quad (t,x') \in Q_{6\rho}'(\hat{z}'), \quad \text{with} \quad \hat{z}' =(\hat{t}, \hat{x}').
 \]
Then, for every $\kappa \in (0,1)$ it holds that
\[
\begin{split}
& \fint_{Q_{\kappa \rho}^+(\hat{z})} | \hat{U} - (\hat{U})_{Q_{\kappa \rho}^+(\hat{z})} | \ \mu(dz) \\
& \leq C(d,\nu, p, q, \alpha)\left[\kappa \left(\fint_{Q_{8\rho}^+(\hat{z})} |\hat{U}|\ \mu(dz) \right) +  \kappa^{-(d+3)} \rho_1^{2(1-1/q)}\left(  \fint_{Q_{8\rho}^+(\hat{z})} | \hat{U} |^q \ \mu(dz) \right)^{\frac{1}{q}} \right ] \\
&  \quad + C(d, \nu, p, q, \alpha) \kappa^{-\frac{d+3}{q}}\left[ a^{\#}_{\rho_0}(\hat{z})^{\frac{1}{q} -\frac{1}{p}} \left( \fint_{Q_{8\rho}^+(\hat{z})} (|D^2u|+|D_d u/x_d|)^p \ \mu(dz) \right)^{1/p} \right. \\
& \qquad +  \left.  a^{\#}_{\rho_0}(\hat{z})^{\frac{1}{q} - \frac{1}{p}}\left( \fint_{Q_{8\rho}^+(\hat{z})}(|u_t|+\lambda |u|)^p \ \mu(dz) \right)^{1/p}  + \left( \fint_{Q_{8\rho}^+(\hat{z})} |f|^q \ \mu(dz) \right)^{1/q} \right],
\end{split}
\]
where $\hat{U} = (\lambda u, u_t, \sqrt{\lambda} Du, DD_{x'} u)$.
\end{lemma}
\begin{proof} We split the proof into two cases depending on $8\rho > \rho_0$ or $8\rho \leq \rho_0$. \\
{\bf Case I}:  $8\rho > \rho_0$. We see that
\[
\begin{split}
 & \fint_{Q_{\kappa \rho}^+(\hat{z})} | \hat{U} - (\hat{U})_{Q_{\kappa \rho}^+(\hat{z})} | \ \mu(dz)\leq   C(d, \alpha) \kappa^{-(d+3)}\fint_{Q_{8\rho}^+(\hat{z})} | \hat{U}| \ \mu(dz) \\
 & \leq C(d, \alpha) \kappa^{-(d+3)} \left(  \fint_{Q_{8 \rho}^+(\hat{z})} \mathbf{1}_{(t_1 -(\rho_0\rho_1)^2, t_1]} \ \mu(dz) \right)^{1-\frac{1}{q}} \left(  \fint_{Q_{8\rho}^+(\hat{z})} | \hat{U} |^q \ \mu(dz) \right)^{\frac{1}{q}} \\
 &\leq C(d, \alpha) \kappa^{-(d+3)} \rho_1^{2(1-1/q)}\left(  \fint_{Q_{8\rho}^+(\hat{z})} | \hat{U} |^q \ \mu(dz) \right)^{\frac{1}{q}}.
 \end{split}
\]
\noindent
{\bf Case 2}: $8\rho \leq \rho_0$. Recall that $[a_{ij}]_{8\rho, \hat{z}}(x_d)$ are defined as in \eqref{average-a} for $i, j \in \{1, 2,\ldots, d\}$.  
Denote
\[
\mathcal{L}_{\rho,\hat{z}} u = [a_{ij}]_{8\rho, \hat{z}}(x_d) D_{ij} u + \frac{\alpha}{x_d} [a_{dd}]_{8\rho, \hat{z}}(x_d) {D_{d}} u.
\]
We also denote
\[
F_1(t,x)  =  (a_{ij} - [a_{ij}]_{8\rho, \hat{z}}) D_{ij} u(t,x), \quad F_2(t,x) = \frac{\alpha}{x_d}(a_{dd} - [a_{dd}]_{8\rho, \hat{z}}) D_{d} u(t,x),
\]
and
$$
F_3(t,x)  =  ([a_{0}]_{8\rho, \hat{z}}-a_{0}) u_t(t,x)
+\lambda ([c_{0}]_{8\rho, \hat{z}}-c_{0}) u(t,x).
$$
Under the condition \eqref{extension-type-matrix},  we observe that $u$ is a solution of
\begin{equation*} 
[a_0]_{8\rho, \hat{z}}(x_d) u_t - \mathcal{L}_{\rho, \hat{z}} u(t,x)  + \lambda  [c_0]_{8\rho, \hat{z}}(x_d) u=  f(t,x) + \sum_{i=1}^3 F_i(t,x)  \quad \text{in} \quad Q_{6\rho}^+(\hat{z}).
\end{equation*}
Moreover, if $\hat{x}_d \leq 6\rho$, the solution $u$ also satisfies the boundary condition
\[
\lim_{x_d \rightarrow 0^+} x_d^\alpha D_d u(t, x', x_d)  = 0 \quad \text{for} \quad (t,z') \in Q_{6\rho}'(\hat{z}').
\]
By applying Corollary \ref{Osc-L-q}, we infer that
\begin{equation} \label{F-oss}
\begin{split}
& \fint_{Q_{\kappa\rho}^+(\hat{z})} |\hat{U}- (\hat{U})_{Q_{\kappa \rho}^+(\hat{z})}| \ \mu(dz)\\
& \leq C( d,\nu, \alpha, q)\left[\kappa  \fint_{Q_{8\rho}^+(\hat{z})} |\hat{U}| \ \mu(dz) + \kappa^{-(d+3)/q}\left( \fint_{Q_{8\rho}^+(\hat{z})} |f(t,x)|^q \ \mu(dz) \right)^{1/q}\right. \\
& \quad + \left.   \kappa^{-(d+3)/q}\sum_{i=1}^3\left( \fint_{Q_{8\rho}^+(\hat{z})} |F_i|^q \ \mu(dz) \right)^{1/q}\right],
\end{split}
\end{equation}
where $\hat{U} = (\lambda u, u_t, \sqrt{\lambda} Du, D D_{x'} u)$. We now control the last three terms on the right hand side of \eqref{F-oss}.  By H\"{o}lder's inequality and the boundedness of $(a_{ij})_{i,j=1}^d$ in \eqref{ellipticity}, we see that
\[
\begin{split}
& \left( \fint_{Q_{8\rho}^+(\hat{z})} |F_1|^q \ \mu(dz) \right)^{1/q} \\
& \leq  \left( \fint_{Q_{8\rho}^+(\hat{z})} |a_{ij}(z) -[a_{ij}]_{8\rho, \hat{z}}(x_d) |^{\frac{pq}{p-q}} \ \mu(dz) \right)^{\frac{p-q}{pq}} \left( \fint_{Q_{8\rho}^+(\hat{z})} |D^2u|^p \ \mu(dz) \right)^{1/p} \\
& \leq  C(\nu, p, q) \left( \fint_{Q_{8\rho}^+(\hat{z})} |a_{ij}(z) -[a_{ij}]_{8\rho, \hat{z}}(x_d) |  \ \mu(dz) \right)^{\frac{p-q}{pq}} \left( \fint_{Q_{8\rho}^+(\hat{z})} |D^2u|^p \ \mu(dz) \right)^{1/p} \\
& = C(\nu, p, q)a^{\#}_{\rho_0}(\hat z)^{\frac{1}{q} - \frac{1}{p}}\left( \fint_{Q_{8\rho}^+(\hat{z})} |D^2u|^p \ \mu(dz) \right)^{1/p}.
\end{split}
\]
Similarly, we use H\"{o}lder's inequality and the boundedness of $a_{ij}$,  $a_0$, and $c_0$ to control the terms  involving $F_2$ and $F_3$ by
\[
 \left( \fint_{Q_{8\rho}^+(\hat{z})} |F_2|^q \ \mu(dz) \right)^{1/q} \leq C(\nu, p, q) a^{\#}_{\rho_0}(\hat z)^{\frac{1}{q} - \frac{1}{p}}\left( \fint_{Q_{8\rho}^+(\hat{z})}\Big|\frac{D_{d}u}{x_d}\Big|^p \ \mu(dz) \right)^{1/p}
\]
and
\[
 \left( \fint_{Q_{8\rho}^+(\hat{z})} |F_3|^q \ \mu(dz) \right)^{1/q} \leq C(\nu, p, q) a^{\#}_{\rho_0}(\hat{z})^{\frac{1}{q} - \frac{1}{p}}\left( \fint_{Q_{8\rho}^+(\hat{z})}(|u_t|+\lambda |u|)^p \ \mu(dz) \right)^{1/p}.
\]
The lemma then follows by combining the above two cases. 
\end{proof}
\begin{proposition}  \label{small-spt.thrm} Let $T, \nu, p, q, K, \alpha$, and $\omega$ be as in Theorem \ref{para-main.theorem}. Then, there exist sufficiently small numbers $\delta = \delta (d, \nu, \alpha, p, q, K) >0$ and $\rho_1 = (d, \nu, \alpha, p, q, K) >0$ such that the following statement holds. Let $\rho_0>0$,
$\lambda >0$, and $f \in L_{q,p}(\Omega_T, {\omega}\ d\mu)$.  Suppose that  \eqref{ellipticity}, \eqref{a-b.zero0}, \eqref{extension-type-matrix}, and \eqref{para-VMO} hold. If $u \in W^{1,2}_{q,p}(\Omega,  \omega\ d\mu)$ vanishes on $(t_1 - (\rho_0 \rho_1)^2, t_1]$ for some $t_1 \in \bR$ and satisfies \eqref{eqn.variable-coeff}, then
\[
\begin{split}
& \|u_t\|_{L_{q,p}(\Omega_T, \omega\,d\mu)} + \sqrt{\lambda} \|Du\|_{L_{q,p}(\Omega_T, \omega\,d\mu)} + \|D^2 u\|_{L_{q,p}(\Omega_T, \omega\,d\mu)} + \lambda \|u\|_{L_{q,p}(\Omega_T, \omega\,d\mu)} \\
& \leq C(d, \nu, \alpha, p, q, K) \|f\|_{L_{q,p}(\Omega_T, \omega\,d\mu)} .
\end{split}
\]
\end{proposition}
\begin{proof}
For the given $\omega_0 \in A_q((-\infty,T))$ and $\omega_1\omega_2 \in A_p(\bR^d_+, d\mu)$, using the reverse H\"older's inequality \cite[Theorem 3.2]{MS1981}
we choose $p_1=p_1(d,p,q,\alpha,K)\in (1,\min(p,q))$ such that
\begin{equation}
							\label{eq0605_13}
\omega_0 \in A_{q/p_1}((-\infty,T)),\quad
\omega_1\omega_2 \in A_{p/p_1}(\bR^d_+, d\mu).
\end{equation}
Let $p_2=(1+p_1)/2\in (1,p_1)$.
Applying Lemma \ref{lemma.3.3} with $p_2, p_1$ in place of $q, p$ respectively, we see that
\[
\begin{split}
\hat{U}^{\#}(\hat{z}) & \leq C(\nu, d, p, p_1, \alpha)\left[\kappa \mathcal{M}(|\hat{U}|)(\hat{z}) + \kappa^{-(d+3)} \rho_1^{2(1-1/p_2)}\mathcal{M}(|\hat{U}|^{p_2})(\hat{z})^{1/p_2} \right. \\
&  \left. +\kappa^{-\frac{d+3}{p_2}} \mathcal{M} (|f|^{p_2})^{1/p_2}(\hat{z}) +  \kappa^{-\frac{d+3}{p_2}} \delta^{\frac{1}{p_2} -\frac{1}{p_1}} \big( \mathcal{M}(|D^2u|^{p_1})^{1/p_1}(\hat{z}) + \mathcal{M}(|D_du/x_d|^{p_1})^{1/p_1}(\hat{z})\big) \right.\\
&\left. +\kappa^{-\frac{d+3}{p_2}} \delta^{\frac{1}{p_2} -\frac{1}{p_1}} \big( \mathcal{M}(|u_t|^{p_1})^{1/p_1}(\hat{z}) + \lambda \mathcal{M}(|u|^{p_1})^{1/p_1}(\hat{z})\big)
\right]
\end{split}
\]
for any $\hat z\in \overline{\Omega_T}$.  Therefore, it follows from the Fefferman-Stein sharp function theorem, Theorem \ref{FS-thm},  that
\[
\begin{split}
& \norm{\hat{U}}_{L_{q,p}(\Omega_T, \omega\,d\mu)} \\
&\leq C(d,\nu, p, q, \alpha, K) \left[ \kappa \| \mathcal{M}(|\hat{U}|)\|_{L_{q,p}(\Omega_T, \omega\,d\mu)}  + \kappa^{-(d+3)} \rho_1^{2(1-1/p_2)}\| \mathcal{M}(|\hat{U}|^{p_2})^{1/p_2}\|_{L_{q,p}(\Omega_T, \omega\,d\mu)}   \right. \\
&  \,\, \left.  + \kappa^{-\frac{d+3}{p_2}} \| \mathcal{M} (|f|^{p_2})^{\frac 1 {p_2}}\|_{L_{q,p}(\Omega_T, \omega\,d\mu)}+  \kappa^{-\frac{d+3}{p_2}} \delta^{\frac{1}{p_2} -\frac{1}{p_1}} \Big\{ \|\mathcal{M}(|D^2u|^{p_1})^{\frac 1 {p_1}}\|_{L_{q,p}(\Omega_T, \omega\,d\mu)}  \right. \\
& \,\, \left. + \|\mathcal{M}(|D_du/x_d|^{p_1})^{\frac 1 {p_1}}\|_{L_{q,p}(\Omega_T, \omega\,d\mu)} +\|\mathcal{M}(|u_t|^{p_1})^{\frac 1 {p_1}}\|_{L_{q,p}(\Omega_T, \omega\,d\mu)}+\lambda \|\mathcal{M}(|u|^{p_1})^{\frac 1 {p_1}}\|_{L_{q,p}(\Omega_T, \omega\,d\mu)} \Big\}
\right].
\end{split}
\]
It then follows from  \eqref{eq0605_13} and the theorem for Hardy-Littlewood maximal function that
\[
\begin{split}
&\norm{\hat{U}}_{L_{q,p}(\Omega_T, w\,d\mu)}  \leq C(d,\nu, p,q, \alpha, K) \left[ \Big(\kappa + \kappa^{-(d+3) }\rho_1^{2(1-1/p_2)}\Big) \| \hat{U}\|_{L_{q,p}(\Omega_T, w\,d\mu)}\right. \\
& \quad \left. + \kappa^{-\frac{d+3}{p_2}} \|f\|_{L_{q,p}(\Omega_T, w\,d\mu)}   +  \kappa^{-\frac{d+3}{p_2}} \delta^{\frac{1}{p_2} -\frac{1}{p_1}} \Big\{ \|D^2u\|_{L_{q,p}(\Omega_T, w\,d\mu)}  + \| D_du/x_d\|_{L_{q,p}(\Omega_T, w\,d\mu)}\right.\\
&\quad \left.+\|u_t\|_{L_{q,p}(\Omega_T, w\,d\mu)}+
\lambda\| u\|_{L_{q,p}(\Omega_T, w\,d\mu)}\Big\}
\right].
\end{split}
\]
By combining this with Lemma \ref{D-dd-control}, we conclude that
\[
\begin{split}
\norm{\tilde{U}}_{L_{q,p}(\Omega_T, w\,d\mu)}&\leq C(d,\nu, p, q,\alpha, K) \left[ \Big (\kappa + {\kappa^{-(d+3)} \rho_1^{2(1-1/p_2)}}+  \kappa^{-\frac{d+3}{p_2}} \delta^{\frac{1}{p_2} -\frac{1}{p_1}}) \| \tilde{ U}\|_{L_{q,p}(\Omega_T, w\,d\mu)} \right.\\
& \quad \qquad \left.+ (\kappa^{-\frac{d+3}{p_2}}+1) \|f\|_{L_{q,p}(\Omega_T, w\,d\mu)} \right],
\end{split}
\]
where 
$$
\tilde{U}=|u_t|+|D^2 u|+\sqrt\lambda |Du|+\lambda |u| +|D_d u/x_d|.
$$
Now, choose $\kappa$ sufficiently small and then choose $\delta$ and $\rho_1$ sufficiently small depending on $d,\nu, p, q,\alpha$, and $K$ such that
\[
C(d,\nu, p,q,\alpha,K)\Big (\kappa + \kappa^{-(d+3)} \rho_1^{2(1-1/p_2)} +  \kappa^{-\frac{d+3}{p_2}} \delta^{\frac{1}{p_2} -\frac{1}{p}}) <1/2.
\]
We then infer that
\[
\begin{split}
& \norm{\tilde{U}}_{L_{q,p}(\Omega_T, w\,d\mu)}   \leq  C(d, \nu, p, q, \alpha, K)  \|f\|_{L_{q,p}(\Omega_T, w\,d\mu)}.
\end{split}
\]
The theorem is proved.
\end{proof}
Now, we are ready to prove Theorem \ref{para-main.theorem}.
\begin{proof}[Proof of Theorem \ref{para-main.theorem}]
We first prove the estimate \eqref{main-para.est}.  Let $u \in W^{1,2}_{q, p}(\Omega_T, \omega\ d\mu)$ be a solution of \eqref{main-eqn}. In view of Proposition \ref{small-spt.thrm}, we only need to remove the condition that $u$ vanishes on $(t_1 - (\rho_0\rho_1)^2, t_1]$.  To this end, we use a partition of unity argument. Take $\xi \in C_0^\infty(\bR)$  to be a non-negative standard cut-off function vanishing outside $(-\rho_0^2\rho_1^2, 0]$ and
\begin{equation} \label{normalized}
\int_{\bR} \xi^q(t)\ dt =1.
\end{equation}
Let $w_s(t,x) = u(t,x)\xi(t-s)$. Then $w_s$ solves the equation
\[
\left\{
\begin{array}{ccc}
 a_0\partial_t w_s -\mathcal{L} w + \lambda  c_0 w & =  & F_s   \\
\displaystyle{\lim_{x_d \rightarrow 0^+} x_d^\alpha D_d w_s(t,x)} & = & 0
\end{array} \quad \text{in} \quad \Omega_T.  \right.
\]
where
$$F_s(t,x) = f (t,x) \xi(t-s) +  a_0 u(t,x) \xi_t (t-s).
$$
Observe that for each $s \in \bR$, the function $w_s$ vanishes outside $(s-\rho_0^2\rho_1^2, s] \times \bR^d_{+}$. We then apply Proposition \ref{small-spt.thrm} to infer that
\begin{equation} \label{w-s.est}
\begin{split}
& \|\partial_t w_s\|_{L_{q,p}(\Omega_T, \omega\, d\mu)} + \sqrt{\lambda} \|Dw_s\|_{L_{q,p}(\Omega_T, \omega\, d\mu)} + \|D^2 w_s\|_{L_{q,p}(\Omega_T,\omega\, d \mu)} + \lambda \|w_s\|_{L_{q,p}(\Omega_T, \omega\, d\mu)} \\
& \leq C(d, \nu, \alpha, p, q, K) \|F_s\|_{L_{q,p}(\Omega_T, \omega\, d\mu)}.
\end{split}
\end{equation}
We note that from \eqref{normalized},
\[
|D^\sigma_x u (t,x)|^p = \int_{\bR}|D_x^\sigma u (t,x )|^p \xi^q(t-s) \ ds, \quad {\forall \ \sigma \in \bN}.
\]
Therefore, with $\tilde{\omega}(x) = \omega_1(x')\omega_2(x_d)$, we have
\[
\begin{split}
\|D^\sigma_x u(t,\cdot)\|_{L_{p}(\bR^d_+, \tilde{\omega}\, d\mu)}^q & =\int_{\bR}\|D_x^\sigma u (t, \cdot)\|^q_{L_{p}(\bR^d_+, \tilde{\omega}\, d\mu)} \xi^q(t-s) \ ds \\
& =\int_{\bR}\|D_x^\sigma [w_s (t, \cdot)]\|^q_{L_{p}(\bR^d_+, \tilde{\omega}\, d\mu)} \ ds.
\end{split}
\]
Then, by integrating the last estimate with respect to $\omega_0(t)\ dt$, we obtain
\[
\|D^\sigma_x u\|_{L_{p, q}(\Omega_T, \omega\, d\mu)}^q = \int_{\bR}\|D_x^\sigma w_s\|^q_{L_{p,q}(\Omega_T, \omega\, d\mu)} \ ds.
\]
Also, since $u_t \xi(t-s) = \partial_t w_s - u \xi_t(t-s)$, by a similar calculation, we also obtain
\[
\|u_t\|^{q}_{L_{q,p}(\Omega_T, \omega\, d\mu)}  \leq  C\int_{\bR} \|\partial_t  w_s\|^q_{L_{p,q}(\Omega_T, \omega\, d\mu)} \ ds + C (\rho_0\rho_1)^{-2 q} \|u\|^{ q}_{L_{q,p}(\Omega_T, \omega\, d \mu)}.
\]
From the last two estimates and by integrating  the $q$-th power of  \eqref{w-s.est} with respect to $ s$, we conclude that
\[
\begin{split}
& \|u_t\|_{L_{q,p}(\Omega_T, \, d\mu)}  + \sqrt{\lambda} \|D u\|_{L_{q,p}(\Omega_T, \omega\, d\mu)} + \|D^2 u\|_{L_{q,p}(\Omega_T, \omega\, d\mu)} + \lambda \|u\|_{L_{q, p}(\Omega_T, \omega\, d\mu)} \\
& \leq C(d, \nu, \alpha, p, q, K) \Big[ \|f\|_{L_{q,p}(\Omega_T, \omega\, d\mu)} +  (\rho_0\rho_1)^{-2} \|u\|_{L_{q,p}(\Omega_T,\omega\, d \mu)}\Big].
\end{split}
\]
Now, choose $\lambda_0 = 2 C(d, \nu, \alpha, p, q, K)(\rho_0\rho_1)^{-2}$, we obtain \eqref{main-para.est} as long as $\lambda \geq \lambda_0$.

It now remains to prove the existence of solutions, as the uniqueness follows from  \eqref{main-para.est}.
We use the method of continuity. To this end, let us fixed $\lambda \geq \lambda_0$ and  for every $\gamma \in [0, 1]$, consider the operator
\[
\begin{split}
\mathcal{L}_\gamma  & = (1-\gamma)\Big[\partial_t - \Delta -\frac{\alpha}{x_d} D_d + \lambda \Big] + \gamma[ a_0\partial_t - \mathcal{L} + \lambda  c_0] \\
& =  (1-\gamma+\gamma a_0) \partial_t  - \Big[(1-\gamma)\Big(\Delta + \frac{\alpha}{x_d}D_d\Big) + \gamma \mathcal{L} \Big] + \lambda  (1-\gamma+\gamma c_0)\\
& =  a_{0,\gamma}\partial_t  - a_{ij,\gamma} D_{ij}u - \frac{\alpha}{x_d} a_{dd,\gamma} D_d u + \lambda c_{0,\gamma},
\end{split}
\]
where
 \[
a_{ij,\gamma} = (1-\gamma)\delta_{ij} + \gamma a_{ij} \quad \text{for} \quad i,j =1,2,\ldots, d,
\]
and
$$
a_{0,\gamma}=1-\gamma+\gamma a_0,\quad
c_{0,\gamma}=1-\gamma+\gamma c_0.
$$
In particular, $(a_{ij,\gamma})_{i,j=1}^d$,  $a_{0,\gamma}$, and $c_{0,\gamma}$ satisfy the same conditions as $(a_{ij})_{i,j=1}^d$,  $a_0$, and $c_0$ do. Therefore, by the a-priori estimate \eqref{main-para.est}, there is a constant $C_0 = C_0(\nu, d, p, q, \alpha, K)$ independent of $\gamma$ such that
\[
\begin{split}
& \norm{u_t}_{L_{q,p}(\Omega_T, \omega\, d\mu)} + \norm{D^2u}_{L_{q,p}(\Omega_T, \omega\, d\mu)} + \sqrt{\lambda} \norm{Du}_{L_{q,p}(\Omega_T, \omega\, d\mu)} + \lambda \norm{u}_{L_{q, p}(\Omega_T, \omega\, d\mu)}\\
& \leq   C_0\norm{f}_{L_{q,p}(\Omega_T, \omega\, d\mu)}
\end{split}
\]
if $u \in W^{1,2}_{q, p}(\Omega_T, \omega\, d\mu)$ is a solution of
\begin{equation} \label{eqn-gamma}
\left\{
\begin{array}{ccc}
 \mathcal{L}_\gamma u  & = & f\\
 \displaystyle{\lim_{x_d \rightarrow 0^+}x_d^\alpha D_d u} & = & 0
\end{array} \quad \text{in} \quad \Omega_T. \right.
\end{equation}
Moreover, observe that when $\gamma =0$,  the existence of solutions to \eqref{eqn-gamma} follows from Theorem \ref{W-2-p.constant.eqn} and the argument in \cite[Section 8]{Dong-Kim-18}. Hence, by the method of continuity, for every $f \in L_{p,q}(\Omega_T, \omega\, d\mu)$, there exists a unique solution $u \in W^{1,2}_{q, p}(\Omega_T, \omega\, d\mu)$ of the equation \eqref{main-eqn}. See, for instance, \cite[Theorem 1.3.4, p. 15]{Krylov-book}. The proof of the theorem is completed.
\end{proof}

\begin{proof}[Proof of Theorem \ref{elli-main.theorem}]
Let $\lambda_0$ and $\delta$ be as in Theorem \ref{para-main.theorem}, and we prove Theorem \ref{elli-main.theorem} with this choice of $\lambda_0$ and $\delta$. It suffices to prove the estimate \eqref{main-ell.est} as the existence and uniqueness can be proved in the same way as in the proof of Theorem \ref{elli-main.theorem}. We follow an approach used in \cite{Krylov-07}. Let $\lambda \geq \lambda_0$ and let $u \in W^{2}_{p}(\bR^d_+, \omega\, d\mu)$ be a solution of the equation \eqref{elli-main-eqn}, and let $v= \xi(t/n)u(x)$ for $n\in \mathbb{N}$ and for some $\xi \in C_0^\infty(\bR)$ satisfying $\xi=1$ on $(0,1)$. Also, let
$$\tilde{f}(t,x) = \xi(t/n) f(x) + \xi'(t/n) u(x)/n. $$
We see that $v \in W^{1,2}_p(\bR^{d+1}_+, \omega\, d\mu)$ is a solution of
\begin{equation*} 
\left\{
\begin{array}{ccc}
v_t - \mathcal{L} v + \lambda  c_0v & = & \tilde{f} \\
\lim_{x_d \rightarrow 0^+} x_d^\alpha D_d v & = & 0
\end{array}\quad \text{in} \quad \bR^{d+1}_+. \right.
\end{equation*}
Then, it follows from Theorem  \ref{para-main.theorem} that
\begin{equation}
                                        \label{eq5.04}
\begin{split}
& \|D^2v\|_{L_p(\bR^{d+1}_+, \omega\, d\mu)} + \sqrt{\lambda} \|D v\|_{L_p(\bR^{d+1}_+, \omega\, d\mu)} + \lambda \|v\|_{L_p(\bR^{d+1}_+, \omega\, d\mu)} \\
&\leq C (\nu, d, p, \alpha, K) \norm{\tilde{f}}_{L_p(\bR^{d+1}_+, \omega\, d\mu)}.
\end{split}
\end{equation}
Let us denote
\[
\gamma_1 = \int_{\bR} |\xi(t)|^p \ dt \quad \text{and} \quad \gamma_2 = \int_{\bR} |\xi'(t)|^p \ dt.
\]
By a simple calculation, we see that
\[
\begin{split}
& \|v\|_{L_p(\bR^{d+1}_+, \omega\, d\mu)}^p = \gamma_1 n\|u\|_{L_p(\bR^{d}_+, \omega\, d\mu)}^p, \quad \|Dv\|_{L_p(\bR^{d+1}_+, \omega\, d\mu)}^p =  \gamma_1 n\|Du\|_{L_p(\bR^{d}_+, \omega\, d\mu)}^p, \\
& \|D^2v\|_{L_p(\bR^{d+1}_+, \omega\, d\mu)}^p =  \gamma_1 n\|D^2u\|_{L_p(\bR^{d}_+, \omega\, d\mu)}^p. 
\end{split}
\]
Moreover,
\[
\norm{\tilde{f}}_{L_p(\bR^{d+1}_+, \omega\, d\mu)}^p \leq \gamma_1 n\norm{f}_{L_p(\bR^{d+1}_+, \omega\, d\mu)}^p + \gamma_2 n^{1-p}\norm{v}_{L_p(\bR^{d+1}_+, \omega\, d\mu)}^p.
\]
By substituting these estimates into \eqref{eq5.04}, we have
\[
\begin{split}
& \|D^2u\|_{L_p(\bR^{d+1}_+, \omega\, d\mu)} + \sqrt{\lambda} \|D u\|_{L_p(\bR^{d+1}_+, \omega\, d\mu)} +  \lambda \|u\|_{L_p(\bR^{d+1}_+, \omega\, d\mu)} \\
&\leq C(\nu, d, p, \alpha, K) \left[  \norm{\tilde{f}}_{L_p(\bR^{d+1}_+, \omega\, d\mu)} + n^{-1}  \|u\|_{L_p(\bR^{d+1}_+, \omega\, d\mu)} \right].
\end{split}
\]
Then, letting $n \rightarrow \infty$, we obtain \eqref{main-ell.est}. The theorem is proved.
\end{proof}

\end{document}